\documentclass[11pt]{article}
\usepackage[T1]{fontenc}
\usepackage[utf8]{inputenc}
\usepackage{lmodern,microtype,amsmath,amssymb,amsthm,mathtools}
\usepackage[margin=1in]{geometry}
\usepackage{enumitem}
\usepackage{xcolor}
\usepackage[colorlinks=true,linkcolor=blue!45!black,citecolor=blue!45!black,urlcolor=blue!45!black]{hyperref}
\hypersetup{pdftitle={Sewing on Thin Groupoids, Knitting, and Based Holonomy for Lipschitz Paths},pdfauthor={Alexandre Reggiolli Teixeira}}
\newtheorem{theorem}{Theorem}[section]
\newtheorem{proposition}[theorem]{Proposition}
\newtheorem{lemma}[theorem]{Lemma}
\newtheorem{corollary}[theorem]{Corollary}
\theoremstyle{definition}
\newtheorem{definition}[theorem]{Definition}
\theoremstyle{remark}
\newtheorem{remark}[theorem]{Remark}
\DeclareMathOperator{\Lip}{Lip}
\DeclareMathOperator{\len}{len}
\DeclareMathOperator{\mesh}{mesh}
\DeclareMathOperator{\Id}{Id}
\DeclareMathOperator{\diam}{diam}
\newcommand{\dist}{\mathsf d}
\newcommand{\Gthin}{\mathcal G_P^{\mathrm{Lip\text{-}th}}}
\newcommand{\HH}{\mathcal H}
\newcommand{\R}{\mathbb R}
\newcommand{\eps}{\varepsilon}
\numberwithin{equation}{section}
\setlist[enumerate]{itemsep=3pt,topsep=5pt}
\allowdisplaybreaks[1]
\title{Sewing on Thin Groupoids, Knitting, and Based Holonomy for Lipschitz Paths}
\author{Alexandre Reggiolli Teixeira\thanks{Institute of Mathematics,
Statistics and Scientific Computing (IMECC), University of Campinas
(Unicamp), Campinas, SP, Brazil.
Email: \href{mailto:a163407@dac.unicamp.br}{\texttt{a163407@dac.unicamp.br}}.
ORCID: \href{https://orcid.org/0000-0002-5598-4372}{0000-0002-5598-4372}.}}
\date{}
\begin{document}
\maketitle
\begin{abstract}
An approximate action of the pair groupoid of a metric space determines
an action of its Lipschitz-thin groupoid on complete extended metric
fibers, uniquely characterized by its local sewing estimate. We prove
this assertion under a superlinear three-point defect estimate and a Lipschitz
bound for finite products. As a result, we obtain a positive proof of
Curry and Manchon's sewing conjecture \cite{CM}, with the necessary
path-scaling correction to its estimate.
Further restriction to isotropy constructs the based holonomy anticipated in
their Remark~4.15 as a representation of based Lipschitz loops modulo
thin equivalence. Rectangular comparison estimates give full relative
Lipschitz-homotopy descent, in particular under their strong knitting
hypothesis and under any three-point estimate of total order greater
than two. We prove the corresponding flatness and basepoint-covariance
statements. An area model shows that the order threshold is sharp.
We also separate the exponential estimate from the general knitting
assumptions and exhibit a compact disk metric for which strong knitting
does not imply invariance under merely continuous homotopy.
As an application, for controlled fields on a metric space, we prove that the classical rough integral descends to the
thin groupoid and admits substitution with a common rough controller.

\end{abstract}

\section{Introduction}

The sewing principle reconstructs an exactly composable object from
increments whose defect has order strictly larger than one, having interesting applications in stochastic analysis and integration theory (see \cite{Gubinelli}): in the
nonlinear setting, approximate flows and their composition estimates
provide a route to differential equations driven by rough paths
\cite{Bailleul}. This analytic construction also raises a geometric
question: which changes of path leave the sewn transport unchanged?
Curry and Manchon \cite{CM} formulated an interval sewing theorem
for maps between varying complete metric fibers and proposed an extension
to metric parameter spaces. This extension asks for an action
of the groupoid of Lipschitz paths modulo Lipschitz-thin equivalence,
starting from maps $\mu_{xy}:M_y\to M_x$ satisfying a
three-point defect estimate and a Lipschitz bound for products indexed
by finite chains. Their Conjecture~4.13 asks whether sewing along
individual paths gives a map depending only on the corresponding
thin class.

On the homotopy side, their Remark~4.15 connects thin descent to
consistent based holonomy and knitting to flatness. Here flatness means
identity transport for every loop admitting a relative Lipschitz
contraction, the conclusion of their knitting theorem under a stronger
four-point estimate.

Our present work is dedicated to proving the structural assertion of Conjecture~4.13 in its
full generality, with the necessary path-scaling correction to its
quantitative estimate (Theorem~\ref{thm:main}). We also develop, as a corollary to the main result, the
homotopy theory of sewn transport along Lipschitz paths, construct the
associated based holonomy, and prove its flatness under the knitting
hypotheses.
In particular, thin descent holds for every three-point order $p>1$, whereas $p>2$
is the sharp universal threshold guaranteeing full relative Lipschitz-homotopy
descent (Theorems~\ref{grid:superquadratic} and~\ref{curvature:sharp}).

Before diving into the full development, we explain the main ideas
and geometric intuition behind the proofs.

First, we point out that our pathwise construction is quantitative in the following sense. Set $p=1+\eps>1$,
$C=\sum_r C_r$, and $K=2^p C\zeta(p)$. For a path $\gamma$, write
$\ell_{s,t}$ for the length of its restriction between $s$ and $t$.
The point-deletion argument gives
\[
 \mathsf d(\phi^\gamma_{s,t},\mu_{\gamma(s)\gamma(t)})
 \le K g(\ell_{s,t})\ell_{s,t}^p.
\]
For a partition $D$ of that oriented interval, let
$\omega_\gamma(D)$ be the largest length of a restriction between
two consecutive partition points. Keeping one global prefix bound
through all deletions gives
\[
 \mathsf d(P_\gamma(D),\phi^\gamma_{s,t})
 \le K g(\ell_{s,t})\ell_{s,t}\omega_\gamma(D)^{p-1}.
\]
In particular, arbitrary partitions with mesh tending to zero have
the same limit. Central cancellation of the products of a partition
and its reverse proves two-sided inverses. These estimates use only
the stated metric bounds, without any continuity assumption on how
the fibers and increments depend on the parameters.

To pass to thin classes, we use metric trees and compose the transports
along their geodesics. For the constant-speed
geodesics $r_{xy}$ put $F_{xy}=\Phi(r_{xy})$. If $m$ is the median
of $x,y,z$, interval composition and reversal yield
\[
 F_{xy}F_{yz}
 =F_{xm}F_{my}F_{ym}F_{mz}=F_{xz}.
\]
Thus the geodesic transports form an exact action of the pair
groupoid. A hybrid-product estimate then identifies the transport of
every Lipschitz tree path with the geodesic transport between its
endpoints. The exact identity is established before arbitrary path
independence is used.

To compare homotopic paths, we use the factorization theorem of
Esmayli--Haj\l asz \cite[Theorem~1.1 and Remark~1.2]{EH}: a
Lipschitz map of a square whose image has zero two-dimensional Hausdorff content
factors through a metric tree. The factorization $H=q\psi$, with $q$
$1$-Lipschitz, preserves the defect constants and chain control under
pullback. For a relative homotopy, vertical lifts may vary. Their
projections are constant, so their transports are identities. Comparing
the two lifted boundary routes gives equality of the horizontal
transports. A thin contraction lies in a Lipschitz path image, which
has zero two-dimensional Hausdorff content. This proves thin descent for every
$p>1$.

The resulting action gives, at each basepoint $a$, a homomorphism
\[
 \operatorname{Hol}_a:
 \pi_{1,\mathrm{thin}}^{\mathrm{Lip}}(P,a)
 \longrightarrow\operatorname{BiLip}_{\mathrm{bd}}(M_a).
\]
The target consists of bi-Lipschitz bijections with bounded displacement.
Partition products determine the representation and its basepoint
covariance. Factorization through $\pi_1^{\mathrm{Lip}}(P,a)$ requires
triviality on the kernel of the thin-to-full quotient. The knitting
estimates below prove this condition.

For that purpose, rectangular comparison measures the difference
between the two two-edge routes across a homotopy cell. A strip
telescoping identity bounds the difference between two row products
by the sum of these cell defects, multiplied only by a horizontal
prefix bound. Vanishing of the total defect is a sufficient criterion.
We derive partition-based and regular-grid conditions directly from
the increments. If the three-point order satisfies $p>2$, comparison
of both cell routes with their common direct increment gives a total
$O(n^{2-p})$ on a regular grid. The strong four-point hypothesis
also gives such decay and yields the knitting theorem with the
original general controls $f$ and $g$.

The distinction between transport and flat transport is familiar in
the differential-geometric theory of functors and connections
\cite{SW}. Here it appears in an elementary translation model. Chord
integrals of $x_1\,dx_2$ have a three-point defect of order two and
sew to path integrals. The positively oriented unit-square loop has
holonomy equal to translation by one although it admits a relative
Lipschitz contraction. Restricted to the unit square, this approximate
action satisfies every weaker admissible gain and admits no
three-point estimate of order greater than two. The
threshold for universal full Lipschitz-homotopy descent is therefore
sharp.

Two quantitative corrections are necessary for comparison with the
displayed formulas of \cite{CM}. Pullback along $\gamma$ changes the
three-point constants to $C_r\Lip(\gamma)^p$, which must occur in
the parameter-based local bound of the conjecture. In the knitting
theorem, an exponential estimate requires a linear bound on the
Lipschitz constants of individual increments. General continuous
$f$ and $g$ suffice for qualitative descent but do not imply that
linear bound. Both distinctions are accompanied by explicit examples,
and the hypotheses of the qualitative theorems are retained.

Finally, contractibility in the flatness statement means relative
Lipschitz contractibility. This follows the homotopy groupoid used
in \cite{CM}. To show that the distinction is substantive for
arbitrary metric targets, we construct a compact metric disk on
which a strong knitting family has nontrivial transport around a
continuously contractible Lipschitz circle. Thus the metric
qualification cannot be removed solely from the topology of the
parameter space.

We finish in Section~\ref{sec:integration} by returning to rough
integration, with an application to Gubinelli's controlled integral
\cite{Gubinelli}. Thin descent assembles the interval integrals of
controlled fields into a single additive functor on the metric thin
groupoid, with estimates expressed through class length. Exact tree
transport then gives primitives and integration along continuous paths
of finite $\rho$-variation with $1\le\rho<3\alpha$, where $\alpha$ is
the regularity exponent of the enhanced field. These paths may have
infinite length. The application also includes joint stability and
substitution with a common rough controller.

We begin the next section with the metric (background) assumptions and composition estimates that
will control partition products and their limits.

\section{Metric fibers and approximate actions}
\label{sec:framework}

Let $P$ be a metric space. Each $M_x$, $x\in P$, carries an extended
metric $d_x:M_x\times M_x\to[0,\infty]$. Its galaxies are the
equivalence classes for $u\sim v$ when $d_x(u,v)<\infty$.
Completeness of $M_x$ means completeness of each galaxy, and continuity
is taken for their disjoint-union topology. All fibers are nonempty.
The empty cases will be recorded after the hypotheses.

For $x,y\in P$, write
\begin{equation}\label{eq:homspace}
 \mathcal C_{xy}=C(M_y,M_x),\qquad
 \dist_{xy}(A,B)=\sup_{v\in M_y}d_x(Av,Bv).
\end{equation}
A finite Lipschitz bound for $A:M_y\to M_x$ means
\[
 d_x(Av,Aw)\le Ld_y(v,w)\qquad(d_y(v,w)<\infty).
\]
Composition is denoted $AB=A\circ B$. Thus maps associated with a
path from $a$ to $b$ have domain $M_b$ and codomain $M_a$.
The supremum formula also defines comparison distances between
arbitrary maps of these types. Completeness is asserted only for
the spaces of continuous maps in \eqref{eq:homspace}.
The map-space facts used below are those of
\cite[Section~2.2]{CM}. Their quantitative form is recorded to fix
the treatment of extended distances and the direction of every product.

\begin{lemma}\label{lem:maps}
The extended metric space $\mathcal C_{xy}$ is complete. For a sequence
in this space,
\begin{equation}\label{eq:lip-limit}
 A_n\longrightarrow A,\quad\sup_n\Lip(A_n)\le L<\infty
 \quad\Longrightarrow\quad\Lip(A)\le L.
\end{equation}
If $A,A':Y\to X$ and $B,B':Z\to Y$ are composable maps with the
finite comparison bounds appearing on the right, then
\begin{align}
 \dist(AB,A'B')&\le\dist(A,A')+\Lip(A')\dist(B,B'),
                    \label{eq:composition}\\
 \dist(AB,A'B)&\le\dist(A,A'),\qquad
 \dist(AB,AB')\le\Lip(A)\dist(B,B').\label{eq:side-composition}
\end{align}
For $A_j,B_j:X_j\to X_{j-1}$ with finite individual comparison
distances and finite Lipschitz constants for all $A$-prefixes, one has
\begin{equation}\label{eq:hybrid}
 \dist(A_1\cdots A_n,B_1\cdots B_n)
 \le\sum_{j=1}^n\Lip(A_1\cdots A_{j-1})\dist(A_j,B_j).
\end{equation}
\end{lemma}
\begin{proof}
Suppose $(A_n)$ is Cauchy. For every $\delta>0$ there exists $N$ with
\[
 \sup_{v\in M_y}d_x(A_nv,A_mv)\le\delta\qquad(n,m\ge N).
\]
For fixed $v$, the sequence $(A_nv)$ is eventually contained in one
complete galaxy. Denote its limit by $Av$. Letting $m\to\infty$
in the displayed inequality gives
\[
 \sup_{v\in M_y}d_x(A_nv,Av)\le\delta\qquad(n\ge N).
\]
For continuity at $v$, fix $\delta>0$ and an index $n$ with
$\dist(A_n,A)<\delta/3$. Continuity of $A_n$ at $v$ gives a
neighborhood $U$ of $v$ in its galaxy such that
$d_x(A_nw,A_nv)<\delta/3$ for $w\in U$. Then
\[
 d_x(Aw,Av)
 \le d_x(Aw,A_nw)+d_x(A_nw,A_nv)+d_x(A_nv,Av)<\delta.
\]
Thus $A\in\mathcal C_{xy}$ and $A_n\to A$ in its supremum metric.
For $d_y(v,w)<\infty$, the same triangle inequality gives
\[
 d_x(Av,Aw)\le2\dist(A,A_n)+Ld_y(v,w).
\]
The limit $n\to\infty$ proves \eqref{eq:lip-limit}.

At each $z\in Z$,
\[
 \begin{aligned}
 d_X(ABz,A'B'z)
 &\le d_X(ABz,A'Bz)+d_X(A'Bz,A'B'z)\\
 &\le\dist(A,A')+\Lip(A')d_Y(Bz,B'z).
 \end{aligned}
\]
Taking the supremum proves \eqref{eq:composition} and its two special
forms \eqref{eq:side-composition}. To obtain \eqref{eq:hybrid}, set
\[
 H_j=A_1\cdots A_jB_{j+1}\cdots B_n\quad(0\le j\le n).
\]
The common suffix restricts the supremum and the common prefix gives
\[
 \dist(H_j,H_{j-1})
 \le\Lip(A_1\cdots A_{j-1})\dist(A_j,B_j).
\]
Sum over $j$. The empty prefix is an identity and has Lipschitz
constant at most one.
\end{proof}

The approximate-action assumptions below are those of
\cite[Definition~4.11]{CM}.

\begin{definition}\label{def:approx}
An approximate action of $P\times P$ on $(M_x)$ is a family
$\mu_{xy}\in\mathcal C_{xy}$ satisfying
\begin{align}
 \mu_{xx}&=\Id_{M_x},\label{eq:diag}\\
 \Lip(\mu_{xy})&\le1+f(d_P(x,y)),\label{eq:f}\\
 \dist_{xy}(\mu_{xy},\mu_{xz}\mu_{zy})
 &\le\sum_{r=1}^N C_r d_P(y,z)^{a_r}d_P(z,x)^{b_r},\label{eq:defect}\\
 \Lip(\mu_{x_0x_1}\cdots\mu_{x_{n-1}x_n})
 &\le g\left(\sum_{j=1}^n d_P(x_{j-1},x_j)\right)\label{eq:chain}
\end{align}
for every triple and every finite chain. Here
\[
 \begin{gathered}
 f:[0,\infty)\to[0,\infty)\text{ is continuous},\qquad f(0)=0,\\
 g:[0,\infty)\to[1,\infty)\text{ is continuous and nondecreasing},\\
 N\ge1,\quad C_r,a_r,b_r>0,\quad a_r+b_r=1+\eps,
 \quad\eps>0.
 \end{gathered}
\]
We use the constants
\begin{equation}\label{eq:constants}
 p=1+\eps,\qquad C=\sum_{r=1}^NC_r,\qquad K=2^pC\zeta(p).
\end{equation}
\end{definition}

The proof of sewing uses \eqref{eq:diag}, \eqref{eq:defect}, and
\eqref{eq:chain}. Condition~\eqref{eq:f} is retained to state the
result on the original class. No condition $g(0)=1$ is imposed.
If an empty fiber $M_x$ is admitted, existence of $\mu_{xy}:M_y\to M_x$
for every $y$ forces every fiber to be empty. That case has the unique
empty action. The case $P=\varnothing$ is also vacuous.

\begin{lemma}\label{lem:finite-defects}
For every $x,y\in P$,
\begin{equation}\label{eq:approx-inverse}
 \dist_{xx}(\mu_{xy}\mu_{yx},\Id_{M_x})\le C d_P(x,y)^p.
\end{equation}
Every finite chain product $\mu_{x_0x_1}\cdots\mu_{x_{n-1}x_n}$ has
finite supremum distance from $\mu_{x_0x_n}$.
\end{lemma}
\begin{proof}
Put the two outside points equal in \eqref{eq:defect} and use
$a_r+b_r=p$ to obtain \eqref{eq:approx-inverse}. For the second
claim let $P_k=\mu_{x_0x_1}\cdots\mu_{x_{k-1}x_k}$ and
$D_k=\dist(P_k,\mu_{x_0x_k})$. Then $D_1=0$ and
\[
 \begin{aligned}
 D_{k+1}
 &\le\dist(P_k\mu_{x_kx_{k+1}},
                 \mu_{x_0x_k}\mu_{x_kx_{k+1}})
   +\dist(\mu_{x_0x_k}\mu_{x_kx_{k+1}},\mu_{x_0x_{k+1}})\\
 &\le D_k+\sum_{r=1}^NC_r
            d_P(x_{k+1},x_k)^{a_r}d_P(x_k,x_0)^{b_r}<\infty.
 \end{aligned}
\]
Induction gives $D_n<\infty$.
\end{proof}

These estimates apply to finite chains. Paths will supply those chains
through partitions, so we next fix the path operations and the thin
equivalence required of sewn transport.

\section{Lipschitz paths and the thin groupoid}
\label{sec:paths-quotient}

For a Lipschitz path $\gamma:[0,1]\to P$, write
\[
 \bar\gamma(u)=\gamma(1-u),\qquad
 \gamma_{s,t}(u)=\gamma(s+(t-s)u),\qquad c_a(u)=a.
\]
If $\gamma:a\to b$ and $\eta:b\to c$, their concatenation is
\begin{equation}\label{eq:concat-formula}
 (\gamma*\eta)(u)=
 \begin{cases}
  \gamma(2u),&0\le u\le\tfrac12,\\
  \eta(2u-1),&\tfrac12\le u\le1.
 \end{cases}
\end{equation}
The two formulas agree at $u=1/2$, and
\[
 \Lip(\gamma*\eta)\le2\max\{\Lip(\gamma),\Lip(\eta)\}.
\]
For an ordered partition $D=(0=t_0<\cdots<t_n=1)$, put
\[
 S_\gamma(D)=\sum_{j=1}^nd_P(\gamma(t_{j-1}),\gamma(t_j)),
 \qquad \len(\gamma)=\sup_D S_\gamma(D).
\]
Refinement increases $S_\gamma(D)$ by the triangle inequality.
Consequently every sampled subchain has length at most
$\len(\gamma)\le\Lip(\gamma)$. For $s,t\in[0,1]$,
\begin{equation}\label{eq:subpath-geometry}
 \ell_{s,t}:=\len(\gamma|_{[s\wedge t,s\vee t]})
 =\len(\gamma_{s,t})\le\Lip(\gamma)|t-s|,
 \qquad\Lip(\gamma_{s,t})\le\Lip(\gamma)|t-s|.
\end{equation}
The length identity is the bijection between partitions induced by
$u\mapsto s+(t-s)u$, in either orientation. The last inequality is
the Lipschitz inequality for this affine map.

We use exactly the finite-chain quotient in
\cite[Definition~4.9]{CM}.

\begin{definition}\label{def:thin}
A relative Lipschitz homotopy from $\gamma_0:a\to b$ to
$\gamma_1:a\to b$ is a Lipschitz map $H:[0,1]^2\to P$ satisfying
\begin{equation}\label{eq:relative-boundary}
 H(s,0)=\gamma_0(s),\quad H(s,1)=\gamma_1(s),\quad
 H(0,t)=a,\quad H(1,t)=b.
\end{equation}
A loop $\lambda:a\to a$ is Lipschitz thin if it has such a homotopy
to $c_a$ with
\begin{equation}\label{eq:thin-image}
 H([0,1]^2)\subseteq\lambda([0,1]).
\end{equation}
Paths $\gamma_0,\gamma_1:a\to b$ are thin equivalent if there is a
finite sequence of Lipschitz paths $\beta_0,\ldots,\beta_m:a\to b$
with
\[
 \beta_0=\gamma_0,\qquad\beta_m=\gamma_1,
 \qquad\beta_{j-1}*\bar\beta_j\text{ thin}\quad(1\le j\le m).
\]
The square has its Euclidean metric. The equivalent $\ell^1$ metric
will be specified when used for homotopy-grid estimates.
\end{definition}

\begin{lemma}\label{lem:thin-calculus}
A spur $\gamma*\bar\gamma$ is thin. If a thin loop is composed
with a nondecreasing Lipschitz map of $[0,1]$ onto itself, it remains
thin. For any Lipschitz $\rho:[0,1]\to[0,1]$ with
$\rho(0)=0$, $\rho(1)=1$, the paths $\gamma$ and
$\gamma\circ\rho$ are related by one generator of thin equivalence.
\end{lemma}
\begin{proof}
Put $\tau(r)=\min\{2r,2-2r\}$. Then
\begin{equation}\label{eq:spur}
 S_\gamma(r,t)=\gamma((1-t)\tau(r))
\end{equation}
satisfies
\[
 \begin{gathered}
 S_\gamma(r,0)=(\gamma*\bar\gamma)(r),\qquad
 S_\gamma(r,1)=\gamma(0),\\
 S_\gamma(0,t)=S_\gamma(1,t)=\gamma(0),\qquad
 S_\gamma([0,1]^2)\subseteq\gamma([0,1]),\\
 d_P(S_\gamma(r,t),S_\gamma(r',t'))
 \le\Lip(\gamma)(2|r-r'|+|t-t'|).
 \end{gathered}
\]
It is the required thin contraction. If $H$ contracts a thin loop
$\lambda$ and $\sigma$ is a nondecreasing Lipschitz surjection
fixing $0,1$, then $H(\sigma(r),t)$ contracts $\lambda\circ\sigma$.
Its image is contained in $\lambda([0,1])=(\lambda\circ\sigma)([0,1])$.

For the last assertion set
\[
 h(r)=\begin{cases}2r,&0\le r\le\tfrac12,\\
                   \rho(2-2r),&\tfrac12\le r\le1.
       \end{cases}
\]
The values at $r=1/2$ are both one, $h(0)=h(1)=0$, and
$h([0,1])=[0,1]$. With $B=2\max\{1,\Lip(\rho)\}$,
\[
 |h(r)-h(r')|\le B|r-r'|,
 \qquad
 |(1-t)h(r)-(1-t')h(r')|\le B|r-r'|+|t-t'|.
\]
The loop $\gamma*\overline{\gamma\circ\rho}$ is $\gamma\circ h$.
Thus $\gamma((1-t)h(r))$ is its relative Lipschitz contraction, and
\[
 \gamma((1-t)h(r))\in\gamma([0,1])
       =(\gamma*\overline{\gamma\circ\rho})([0,1]).
\]
\end{proof}

\begin{lemma}\label{lem:quotient}
Thin equivalence is a congruence for concatenation. Its classes form
a groupoid $\Gthin$ with
\begin{equation}\label{eq:groupoid-laws}
 [\gamma][\eta]=[\gamma*\eta],\qquad
 1_a=[c_a],\qquad [\gamma]^{-1}=[\bar\gamma].
\end{equation}
The source of $[\gamma]$ is the final point of its representative
and its range is the initial point. Every endpoint-preserving
Lipschitz reparameterization has the same thin class.
\end{lemma}
\begin{proof}
The generator is reflexive by \eqref{eq:spur}. If
$H(r,t)$ contracts $\alpha*\bar\beta$, then $H(1-r,t)$ contracts
$\beta*\bar\alpha$. Thus the generator is symmetric, and its
finite-chain closure is an equivalence relation.

Fixed finite path words may be given any positive time allocations.
For two allocations $(u_j)$ and $(v_j)$, the map that sends each
$u_j$ to $v_j$ and is affine between successive breakpoints is a
Lipschitz increasing homeomorphism. Lemma~\ref{lem:thin-calculus}
identifies the associated classes. Composing a thin contraction with
this homeomorphism preserves thinness itself.

Suppose $\lambda=\alpha*\bar\beta$ is thin, where
$\alpha,\beta:a\to b$, and let $\eta:b\to c$.
On a four-stage allocation, the loop
$(\alpha*\eta)*\overline{\beta*\eta}$ has factors
$\alpha,\eta,\bar\eta,\bar\beta$. Its first homotopy is
\begin{equation}\label{eq:right-whisker}
 U(r,t)=\begin{cases}
  \alpha(4r),&0\le r\le\tfrac14,\\
  S_\eta(2r-\tfrac12,t),&\tfrac14\le r\le\tfrac34,\\
  \bar\beta(4r-3),&\tfrac34\le r\le1.
 \end{cases}
\end{equation}
The formulas coincide at the interfaces. At $t=1$, $U(r,1)$ equals
$\lambda(\sigma(r))$, where
\[
 \sigma(r)=\begin{cases}
 2r,&0\le r\le\tfrac14,\\
 \tfrac12,&\tfrac14\le r\le\tfrac34,\\
 2r-1,&\tfrac34\le r\le1.
 \end{cases}
\]
If $H$ contracts $\lambda$, the map
\[
 V(r,t)=\begin{cases}
 U(r,2t),&0\le t\le\tfrac12,\\
 H(\sigma(r),2t-1),&\tfrac12\le t\le1
 \end{cases}
\]
contracts the original loop. Its image lies in
$\alpha([0,1])\cup\beta([0,1])\cup\eta([0,1])$, exactly the
image of that loop.

For left composition let $\xi:d\to a$. The corresponding loop
has word $\xi,\alpha,\bar\beta,\bar\xi$. Replace the middle
part in \eqref{eq:right-whisker} by $H(2r-1/2,t)$ and the outside
parts by $\xi(4r)$ and $\bar\xi(4r-3)$. At the end this is
$\xi*\bar\xi$ with a constant pause, which contracts by
\eqref{eq:spur} and Lemma~\ref{lem:thin-calculus}. Its image
throughout is contained in the original four-stage loop image.
These piecewise maps are Lipschitz: a segment in the square is split
at the finitely many rectangular interfaces, the formulas agree
there, and their Lipschitz bounds add along the segment. They also
have the fixed boundary values prescribed in
\eqref{eq:relative-boundary}.

We have proved left and right compatibility for each generator.
Applying these operations to every link proves compatibility for
thin chains. The two bracketings of three paths are two positive
allocations of the same three-stage word. They therefore have the
same class. Unit laws are constant pauses, covered by
Lemma~\ref{lem:thin-calculus}. Both spurs contract, and their
constant pauses contract, so
\[
 [\gamma][\bar\gamma]=[c_{\gamma(0)}],\qquad
 [\bar\gamma][\gamma]=[c_{\gamma(1)}].
\]
An associative category in which every morphism has a two-sided
inverse is a groupoid, yielding \eqref{eq:groupoid-laws}.
\end{proof}

\begin{corollary}\label{cor:subpath-class}
For every $s,u,t\in[0,1]$,
\begin{equation}\label{eq:subpath-class}
 [\gamma_{s,t}]=[\gamma_{s,u}][\gamma_{u,t}].
\end{equation}
\end{corollary}
\begin{proof}
When $u$ lies between $s,t$, the concatenation is an increasing
piecewise affine reparameterization of $\gamma_{s,t}$. If $t$ lies
between $s,u$, the ordered identity gives
$[\gamma_{s,u}]=[\gamma_{s,t}][\gamma_{t,u}]$, and multiplication
on the right by $[\gamma_{u,t}]$ gives \eqref{eq:subpath-class}.
If $s$ lies between $u,t$, use
$[\gamma_{u,t}]=[\gamma_{u,s}][\gamma_{s,t}]$ and multiply on the
left by $[\gamma_{s,u}]$. Coincident parameters give identity or
inverse identities in \eqref{eq:groupoid-laws}.
\end{proof}

The path quotient is now defined. To continue our path to the proof of the full conjecture, we now construct transport along its
representatives by taking partition limits, and prove that the limits
respect composition and reversal.

\section{Pathwise sewing}
\label{sec:path}

The point-deletion argument for interval sewing
\cite[Theorem~3.6 and Lemmas~3.7--3.8]{CM} admits estimates in terms
of the length of each restricted path. We give the argument for both
orientations and retain one global Lipschitz bound during refinement.
The constants $p$, $C$, and $K$ are those of \eqref{eq:constants}.

\begin{lemma}[Additive length control]\label{lem:length-control}
Fix a Lipschitz path $\gamma:[0,1]\to P$ and write $L=\Lip(\gamma)$.
For $a<b$, put $v(a,b)=\len(\gamma|_{[a,b]})$, and set $v(a,a)=0$.
For $a\le b\le c$,
\begin{equation}\label{eq:length-additivity}
 v(a,c)=v(a,b)+v(b,c),\qquad
 d_P(\gamma(a),\gamma(b))\le v(a,b)\le L(b-a).
\end{equation}
Consequently $V(t)=v(0,t)$ is nondecreasing and $L$-Lipschitz, and
\[
 \ell_{s,t}=|V(t)-V(s)|,\qquad
 d_P(\gamma(s),\gamma(t))\le|V(t)-V(s)|.
\]
\end{lemma}
\begin{proof}
Every polygonal sum on $[a,b]$ is bounded by
$L\sum_j(t_j-t_{j-1})=L(b-a)$, and the endpoint subdivision gives
the lower bound in \eqref{eq:length-additivity}.
For any subdivision of $[a,c]$, insertion of $b$ cannot decrease
its polygonal sum, by the triangle inequality. The resulting sum
splits at $b$ and is at most $v(a,b)+v(b,c)$.
Taking the supremum gives $v(a,c)\le v(a,b)+v(b,c)$.
Conversely, let $\eta>0$. By the definition of these finite suprema,
there are subdivisions of $[a,b]$ and $[b,c]$ with polygonal sums
greater than $v(a,b)-\eta$ and $v(b,c)-\eta$, respectively.
Their union is a subdivision of $[a,c]$, so
$v(a,c)\ge v(a,b)+v(b,c)-2\eta$.
Since this holds for every $\eta>0$, additivity holds.
For $s\le t$ it gives $V(t)-V(s)=v(s,t)$, which belongs to
$[0,L(t-s)]$. The conclusions for arbitrary $s,t$ use symmetry.
\end{proof}

An oriented subdivision $D=(t_0,\ldots,t_n)$ from $s$ to $t$
has increasing points when $s<t$ and decreasing points when $s>t$.
For $s\ne t$ define
\begin{align*}
 P_\gamma(D)&=\mu_{\gamma(t_0)\gamma(t_1)}\cdots
                 \mu_{\gamma(t_{n-1})\gamma(t_n)},\qquad
 \mesh(D)=\max_j|t_j-t_{j-1}|,\\
 \lambda_j(D)&=|V(t_j)-V(t_{j-1})|,\qquad
 \omega_\gamma(D)=\max_j\lambda_j(D).
\end{align*}
The quantity $\omega_\gamma(D)$ is the largest length of a path
restriction between consecutive subdivision points. Additivity gives
\begin{equation}\label{eq:length-mesh}
 \sum_{j=1}^n\lambda_j(D)=\ell_{s,t},\qquad
 \omega_\gamma(D)\le L\mesh(D).
\end{equation}
Refinement cannot increase $\omega_\gamma(D)$.
For $s=t$ there is one degenerate subdivision, whose product is
$\Id_{M_{\gamma(s)}}$ and whose two meshes are zero.

\begin{lemma}[Deletion and refinement]\label{lem:intrinsic-refinement}
For every oriented subdivision $D$ from $s$ to $t$,
\begin{equation}\label{eq:delete}
 \dist_{\gamma(s)\gamma(t)}
 (P_\gamma(D),\mu_{\gamma(s)\gamma(t)})
 \le K g(\ell_{s,t})\ell_{s,t}^{p}.
\end{equation}
If $D'$ refines $D$, then
\begin{equation}\label{eq:refine}
 \begin{split}
 \dist_{\gamma(s)\gamma(t)}(P_\gamma(D'),P_\gamma(D))
 &\le K g(\ell_{s,t})\sum_{j=1}^n\lambda_j(D)^p\\
 &\le K g(\ell_{s,t})\ell_{s,t}\omega_\gamma(D)^{p-1}.
 \end{split}
\end{equation}
\end{lemma}
\begin{proof}
Write $\ell=\ell_{s,t}$ and $G=g(\ell)$.
If $\ell=0$, Lemma~\ref{lem:length-control} makes $\gamma$ constant
on the interval in question, and every product is an identity.
Suppose $\ell>0$. Every monotone subchain in this interval has
image chain length at most $\ell$, by \eqref{eq:length-additivity}.
Thus every prefix occurring after any sequence of deletions has
Lipschitz constant at most $G$, by \eqref{eq:chain}.
For an empty prefix the same bound holds since $1\le g(0)\le G$.

Consider a current subdivision $U=(u_0,\ldots,u_m)$ with $m\ge2$,
and put $q_j=|V(u_j)-V(u_{j-1})|$.
Since $\sum_{j=1}^m q_j=\ell$,
\[
 \sum_{j=1}^{m-1}(q_j+q_{j+1})\le2\ell.
\]
Let $j$ be the smallest index minimizing $q_j+q_{j+1}$.
Then $q_j+q_{j+1}\le2\ell/(m-1)$.
Deletion of $u_j$ replaces the consecutive factors
$\mu_{\gamma(u_{j-1})\gamma(u_j)}
 \mu_{\gamma(u_j)\gamma(u_{j+1})}$
by $\mu_{\gamma(u_{j-1})\gamma(u_{j+1})}$.
The defect of this replacement is bounded by
\[
 \sum_r C_r q_{j+1}^{a_r}q_j^{b_r}
 \le C(q_j+q_{j+1})^p
 \le C\left(\frac{2\ell}{m-1}\right)^p.
\]
The common right suffix does not increase the comparison distance.
The common left prefix increases it by at most $G$.
Repeating until only the two endpoints remain therefore gives
\[
 \dist(P_\gamma(D),\mu_{\gamma(s)\gamma(t)})
 \le 2^p CG\ell^p\sum_{k=1}^{n-1}k^{-p}
 \le KG\ell^p,
\]
which proves \eqref{eq:delete}. Only absolute increments of $V$
occur, so this calculation covers both orientations.

For refinement, retain the points of $D$ and delete the other points
of $D'$ in the successive coarse intervals.
Inside the $j$th coarse interval the total length is
$\lambda_j(D)$, by additivity. If $m$ subintervals remain in that
block, an interior point has adjacent length sum at most
$2\lambda_j(D)/(m-1)$.
Its removal has error at most
$GC(2\lambda_j(D)/(m-1))^p$ in the full product.
The prefix is still a monotone subchain of the original interval,
even after deletions in other blocks, so the same $G$ applies.
Summing first within each block and then over all blocks gives
the first inequality in \eqref{eq:refine}.
The second is the consequence
\[
 \sum_j\lambda_j(D)^p
 \le\omega_\gamma(D)^{p-1}\sum_j\lambda_j(D)
 =\ell\omega_\gamma(D)^{p-1}.
\]
\end{proof}

\begin{lemma}[Uniform mesh limit]\label{lem:mesh-limit}
For each $s,t$, the products $P_\gamma(D)$ converge in
$\mathcal C_{\gamma(s)\gamma(t)}$ as $\omega_\gamma(D)\to0$.
Write their limit as $\phi^\gamma_{s,t}$.
In particular, convergence holds whenever $\mesh(D)\to0$.
For every oriented subdivision $D$,
\begin{align}
 \dist(\phi^\gamma_{s,t},\mu_{\gamma(s)\gamma(t)})
 &\le Kg(\ell_{s,t})\ell_{s,t}^{p}
 \le KL^p g(\ell_{s,t})|t-s|^p,\label{eq:local}\\
 \dist(P_\gamma(D),\phi^\gamma_{s,t})
 &\le Kg(\ell_{s,t})\sum_j\lambda_j(D)^p
 \le Kg(\ell_{s,t})\ell_{s,t}\omega_\gamma(D)^{p-1},
 \label{eq:intrinsic-rate}\\
 \dist(P_\gamma(D),\phi^\gamma_{s,t})
 &\le Kg(\ell_{s,t})\ell_{s,t}(L\mesh(D))^{p-1}
 \le KL^p g(\ell_{s,t})|t-s|\mesh(D)^{p-1},\label{eq:rate}\\
 \Lip(\phi^\gamma_{s,t})&\le g(\ell_{s,t}).\label{eq:path-lip}
\end{align}
The distances in these inequalities are taken in
$\mathcal C_{\gamma(s)\gamma(t)}$.
\end{lemma}
\begin{proof}
The zero-length case consists of identity products, so assume
$\ell=\ell_{s,t}>0$ and put $G=g(\ell)$.
The ordered union $R$ of two subdivisions $D,E$ refines both.
By \eqref{eq:refine} and the triangle inequality,
\begin{equation}\label{eq:cauchy}
 \dist(P_\gamma(D),P_\gamma(E))
 \le KG\ell\bigl(\omega_\gamma(D)^{p-1}
                      +\omega_\gamma(E)^{p-1}\bigr).
\end{equation}
Let $D_k$ be the oriented subdivision with $2^k$ equal parameter
intervals. Then $\omega_\gamma(D_k)\le L|t-s|2^{-k}$.
Equation~\eqref{eq:cauchy} makes $(P_\gamma(D_k))_k$ Cauchy.
By \eqref{eq:delete}, all its terms belong to the galaxy of
$\mu_{\gamma(s)\gamma(t)}$ in the complete map space.
Their limit $\phi^\gamma_{s,t}$ is therefore defined there.
For fixed $D$, letting $k\to\infty$ in \eqref{eq:cauchy} gives
\[
 \dist(P_\gamma(D),\phi^\gamma_{s,t})
 \le KG\ell\omega_\gamma(D)^{p-1}.
\]
This proves convergence uniformly over all subdivisions with
sufficiently small length mesh. It proves parameter mesh convergence
by \eqref{eq:length-mesh} as well.

For the sharper sum in \eqref{eq:intrinsic-rate}, keep $D$ fixed
and use its ordered unions $R_k$ with $D_k$.
Then $\omega_\gamma(R_k)\le\omega_\gamma(D_k)\to0$ and
$R_k$ refines $D$. Passing to the limit in
\eqref{eq:refine} yields the first inequality in
\eqref{eq:intrinsic-rate}. Its second inequality was proved in the
deletion lemma. Equation~\eqref{eq:local} follows from
\eqref{eq:delete} and $\ell\le L|t-s|$.
The two inequalities in \eqref{eq:rate} use
$\omega_\gamma(D)\le L\mesh(D)$ and $\ell\le L|t-s|$.

Every $P_\gamma(D_k)$ has Lipschitz constant at most $G$.
For $v,w\in M_{\gamma(t)}$ at finite distance, put
$\delta_k=\dist(P_\gamma(D_k),\phi^\gamma_{s,t})$.
The triangle inequality gives
\[
 d_{\gamma(s)}(\phi^\gamma_{s,t}(v),\phi^\gamma_{s,t}(w))
 \le2\delta_k+Gd_{\gamma(t)}(v,w).
\]
Since $\delta_k\to0$, this proves \eqref{eq:path-lip}.
\end{proof}

\begin{proposition}[Two-sided pathwise sewing]\label{prop:path}
The maps of Lemma~\ref{lem:mesh-limit} satisfy its local comparison,
mesh convergence, and Lipschitz estimates. For every $s,u,t\in[0,1]$,
\begin{equation}\label{eq:fullflow}
 \phi^\gamma_{s,s}=\Id_{M_{\gamma(s)}},\qquad
 \phi^\gamma_{s,t}=\phi^\gamma_{s,u}\phi^\gamma_{u,t},\qquad
 \phi^\gamma_{t,s}=(\phi^\gamma_{s,t})^{-1}.
\end{equation}
\end{proposition}
\begin{proof}
The degenerate subdivision gives the diagonal identity.
Suppose first that $u$ lies between $s$ and $t$.
Subdivisions of the two consecutive oriented subintervals concatenate,
and the product of the concatenated subdivision is the product of
their two products. Both subinterval lengths are at most
$\ell_{s,t}$, so their limit maps have Lipschitz constants at most
$G=g(\ell_{s,t})$.
For products $P_k,Q_k$ on parameter mesh-null subdivisions of these
subintervals, \eqref{eq:composition} gives
\[
 \dist(P_kQ_k,\phi^\gamma_{s,u}\phi^\gamma_{u,t})
 \le\dist(P_k,\phi^\gamma_{s,u})
       +G\dist(Q_k,\phi^\gamma_{u,t})\longrightarrow0.
\]
The concatenated products also converge to $\phi^\gamma_{s,t}$.
Separation proves composition for these ordered triples.

Fix a subdivision $D=(t_0,\ldots,t_n)$ from $s$ to $t$, and put
$x_j=\gamma(t_j)$, $P=P_\gamma(D)$, and $Q=P_\gamma(D^{-1})$,
where $D^{-1}$ is the reversed subdivision.
For $0\le j\le n$ define
\[
 T_j=(\mu_{x_0x_1}\cdots\mu_{x_{j-1}x_j})
     (\mu_{x_jx_{j-1}}\cdots\mu_{x_1x_0}),\qquad T_0=\Id.
\]
Thus $T_n=PQ$. Applying \eqref{eq:defect} with equal outer points
gives
\[
 \dist_{x_{j-1}x_{j-1}}
 (\mu_{x_{j-1}x_j}\mu_{x_jx_{j-1}},\Id)
 \le C d_P(x_{j-1},x_j)^p.
\]
The common left prefix in $T_j$ and $T_{j-1}$ has chain length
at most $\ell_{s,t}$, and the common right suffix does not increase
distance. Therefore
\begin{align*}
 \dist_{x_0x_0}(PQ,\Id)
 &\le\sum_{j=1}^n\dist(T_j,T_{j-1})
 \le CG\sum_{j=1}^n d_P(x_{j-1},x_j)^p\\
 &\le CG\sum_{j=1}^n\lambda_j(D)^p
 \le CG\ell_{s,t}\omega_\gamma(D)^{p-1}.
\end{align*}
The reversed tuple $x_n,\ldots,x_0$ gives
the same bound for $\dist_{x_nx_n}(QP,\Id)$.
Moreover, \eqref{eq:composition} and \eqref{eq:path-lip} give
\[
 \dist(PQ,\phi^\gamma_{s,t}\phi^\gamma_{t,s})
 \le\dist(P,\phi^\gamma_{s,t})
       +G\dist(Q,\phi^\gamma_{t,s})\longrightarrow0
\]
as the mesh tends to zero, and the reversed estimate applies to $QP$.
The two triangle inequalities now give
\[
 \phi^\gamma_{s,t}\phi^\gamma_{t,s}=\Id_{M_{\gamma(s)}},\qquad
 \phi^\gamma_{t,s}\phi^\gamma_{s,t}=\Id_{M_{\gamma(t)}}.
\]

To obtain unrestricted composition, set $F_t=\phi^\gamma_{0,t}$.
For $s\le t$, ordered composition gives $F_t=F_s\phi^\gamma_{s,t}$,
and hence $\phi^\gamma_{s,t}=F_s^{-1}F_t$.
For $s\ge t$, inversion and the ordered identity give
$\phi^\gamma_{s,t}=(F_t^{-1}F_s)^{-1}=F_s^{-1}F_t$.
Consequently, for arbitrary $s,u,t$,
\[
 \phi^\gamma_{s,u}\phi^\gamma_{u,t}
 =F_s^{-1}F_uF_u^{-1}F_t
 =F_s^{-1}F_t=\phi^\gamma_{s,t}.
\]
This includes coincident parameters.
\end{proof}

\begin{proposition}[Subpaths and path operations]\label{prop:path-algebra}
Define $\Phi_\mu(\gamma)=\phi^\gamma_{0,1}$. Then
\begin{equation}\label{eq:path-algebra}
 \begin{gathered}
 \Phi_\mu(c_a)=\Id_{M_a},\qquad
 \Phi_\mu(\gamma*\eta)=\Phi_\mu(\gamma)\Phi_\mu(\eta),\\
 \Phi_\mu(\bar\gamma)=\Phi_\mu(\gamma)^{-1},\qquad
 \Phi_\mu(\gamma_{s,t})=\phi^\gamma_{s,t}.
 \end{gathered}
\end{equation}
If $\rho:[0,1]\to[0,1]$ is nondecreasing and Lipschitz, with
$\rho(0)=0$ and $\rho(1)=1$, then
$\Phi_\mu(\gamma\circ\rho)=\Phi_\mu(\gamma)$.
\end{proposition}
\begin{proof}
For a constant path all factors are diagonal identities.
For $s\ne t$, the affine map $u\mapsto s+(t-s)u$ sends every
subdivision $D$ of $[0,1]$ to an oriented subdivision $D_{s,t}$
from $s$ to $t$, with
\[
 P_{\gamma_{s,t}}(D)=P_\gamma(D_{s,t}),\qquad
 \mesh(D_{s,t})=|t-s|\mesh(D).
\]
Both sides therefore have the same limit.
For $s=t$ the subpath is constant, which proves the affine identity
in all cases.
For $\lambda=\gamma*\eta$, the identities
$\lambda_{0,1/2}=\gamma$ and $\lambda_{1/2,1}=\eta$ are exact.
Composition at the midpoint of $\lambda$ yields
\[
 \Phi_\mu(\lambda)
 =\phi^\lambda_{0,1/2}\phi^\lambda_{1/2,1}
 =\Phi_\mu(\gamma)\Phi_\mu(\eta).
\]
Since $\bar\gamma=\gamma_{1,0}$, the affine identity and
\eqref{eq:fullflow} give the reversal identity.

Finally, the continuous nondecreasing map $\rho$ is onto $[0,1]$.
For $D=(t_0,\ldots,t_n)$, delete repetitions in
$(\rho(t_0),\ldots,\rho(t_n))$ and denote the resulting subdivision
by $\rho_\#D$.
Every repeated entry contributes a diagonal identity factor, so
$P_{\gamma\circ\rho}(D)=P_\gamma(\rho_\#D)$.
If $a<b$ are successive distinct entries in $\rho_\#D$, let $j$
be the first index of the original list with $\rho(t_j)=b$.
Then $\rho(t_{j-1})=a$, and
\[
 b-a\le\Lip(\rho)(t_j-t_{j-1})\le\Lip(\rho)\mesh(D).
\]
It follows that $\mesh(\rho_\#D)\le\Lip(\rho)\mesh(D)$.
The two partition limits coincide, proving reparameterization invariance.
\end{proof}

When the approximate action is fixed, we also write $\Phi=\Phi_\mu$.

Transport now respects composition and reversal. To compare different
paths, we turn to metric trees, where the median of three points gives
an exact identity for geodesic transport.

\section{Exact geodesic transport on a tree}
\label{sec:tree}

We recall that a metric tree is a geodesic metric space with a unique arc between
every pair of points. Write $[x,y]$ for its geodesic segment and
$r_{xy}:[0,1]\to T$ for the constant-speed parameterization, with
$r_{xx}=c_x$. The median property states that
\[
 [x,y]\cap[y,z]\cap[z,x]=\{m\},\qquad
 [x,y]=[x,m]\cup[m,y],
\]
with the corresponding decompositions of the other two segments.
This standard tripod description is recalled in
\cite[Section~2.3]{EH}. No completeness assumption on $T$ is imposed.

\begin{theorem}\label{thm:tree}
Let $\nu$ be an approximate action on a metric tree $T$, with the
constants $p,C,K$ of \eqref{eq:constants}. The maps
$F_{xy}=\Phi_\nu(r_{xy})$ form an exact action of the pair groupoid
of $T$ and satisfy
\begin{equation}\label{eq:geo-bound}
 \dist_{xy}(F_{xy},\nu_{xy})
 \le Kg(d_T(x,y))d_T(x,y)^p,
 \qquad \Lip(F_{xy})\le g(d_T(x,y)).
\end{equation}
For every Lipschitz path $\alpha:[0,1]\to T$,
\begin{equation}\label{eq:tree-endpoints}
 \Phi_\nu(\alpha)=F_{\alpha(0)\alpha(1)}.
\end{equation}
\end{theorem}
\begin{proof}
Both the length and the Lipschitz constant of $r_{xy}$ are
$d_T(x,y)$. Proposition~\ref{prop:path} gives
\eqref{eq:geo-bound}. If $x\ne y$ and $v\in[x,y]$, set
$\theta=d_T(x,v)/d_T(x,y)$. For $0\le u\le1$, uniqueness and
constant speed of geodesics give
\[
 r_{xy}(\theta u)=r_{xv}(u),\qquad
 r_{xy}(\theta+(1-\theta)u)=r_{vy}(u).
\]
The formulas hold also for $\theta=0,1$. The affine-subpath and
local-flow identities imply
\begin{equation}\label{eq:geo-split}
 F_{xy}=F_{xv}F_{vy},\qquad
 F_{yx}=F_{xy}^{-1},\qquad F_{xx}=\Id_{M_x}.
\end{equation}
The splitting identity for $x=y$ consists of identity maps.
For the median $m$ of an arbitrary triple $x,y,z$, these identities
give
\begin{equation}\label{eq:median}
 \begin{aligned}
 F_{xy}F_{yz}
 &=F_{xm}F_{my}F_{ym}F_{mz}\\
 &=F_{xm}\Id_{M_m}F_{mz}=F_{xz}.
 \end{aligned}
\end{equation}
This proves the exact action law using only transports on geodesics.

Let $\alpha:a\to b$ be Lipschitz and put $\ell=\len(\alpha)$.
For a partition $D=(0=t_0<\cdots<t_n=1)$, set
\[
 x_j=\alpha(t_j),\qquad d_j=d_T(x_{j-1},x_j),\qquad
 \delta_D=\max_jd_j.
\]
The length and Lipschitz estimates give
\[
 \sum_{j=1}^nd_j\le\ell,\qquad
 \delta_D\le\Lip(\alpha)\mesh(D).
\]
Write $A_j=\nu_{x_{j-1}x_j}$ and $B_j=F_{x_{j-1}x_j}$.
By \eqref{eq:median}, $B_1\cdots B_n=F_{ab}$.
For $0\le j\le n$, define
\[
 H_j=A_1\cdots A_jB_{j+1}\cdots B_n.
\]
Thus $H_0=F_{ab}$ and $H_n=P_\nu(\alpha,D)$.
The comparison of $H_j$ and $H_{j-1}$ has common left prefix
$A_1\cdots A_{j-1}$ and common right suffix $B_{j+1}\cdots B_n$.
The prefix satisfies
\[
 \Lip(A_1\cdots A_{j-1})
 \le g\left(\sum_{i=1}^{j-1}d_i\right)\le g(\ell).
\]
For the empty prefix the bound follows from $1\le g(\ell)$.
The right suffix does not increase the supremum distance. Therefore
\begin{align}
 \dist_{ab}(P_\nu(\alpha,D),F_{ab})
 &\le \sum_{j=1}^n\dist_{ab}(H_j,H_{j-1})\notag\\
 &\le g(\ell)\sum_{j=1}^n
       \dist_{x_{j-1}x_j}(A_j,B_j)\notag\\
 &\le Kg(\ell)g(\delta_D)\sum_{j=1}^nd_j^p\notag\\
 &\le Kg(\ell)^2\ell\delta_D^{p-1}.
 \label{eq:tree-rate}
\end{align}
If $\ell=0$, all factors are identities. Otherwise
$\delta_D\le\ell$ and $p>1$, so the last quantity tends to zero
as $\mesh(D)\to0$. Proposition~\ref{prop:path} also gives
$P_\nu(\alpha,D)\to\Phi_\nu(\alpha)$.
For any $\eta>0$, sufficiently small mesh makes both distances
from $P_\nu(\alpha,D)$ to these two limits less than $\eta/2$.
The triangle inequality and separation in the extended supremum
metric prove \eqref{eq:tree-endpoints}.
\end{proof}

The exact tree action is already determined by a local superlinear
comparison with the approximate action. A separate Lipschitz bound
on products of the competing exact action is unnecessary.

\begin{proposition}\label{prop:tree-uniqueness}
Let $E_{xy}\in\mathcal C_{xy}$ be an exact action of the pair
groupoid of $T$. Suppose that some $A<\infty$, $r_0>0$, and
$\sigma>0$ satisfy
\[
 \dist_{xy}(E_{xy},\nu_{xy})
 \le A d_T(x,y)^{1+\sigma}
 \qquad\bigl(d_T(x,y)\le r_0\bigr).
\]
Then $E_{xy}=F_{xy}$ for all $x,y\in T$.
\end{proposition}
\begin{proof}
Fix $x\ne y$, set $h=d_T(x,y)$, and put
$D_n=(j/n)_{j=0}^n$ and $x_j=r_{xy}(j/n)$.
For every integer $n\ge h/r_0$, the exact
action identity and \eqref{eq:hybrid}, with $\nu$ prefixes, give
\[
 \begin{aligned}
 \dist_{xy}\bigl(P_\nu(r_{xy},D_n),E_{xy}\bigr)
 &\le g(h)\sum_{j=1}^n
       \dist_{x_{j-1}x_j}(\nu_{x_{j-1}x_j},E_{x_{j-1}x_j})\\
 &\le A g(h)n(h/n)^{1+\sigma}
   =A g(h)h^{1+\sigma}n^{-\sigma}\longrightarrow0.
 \end{aligned}
\]
The same partition products converge to $F_{xy}$.
Separation identifies the limits. The diagonal case is the exact
action identity $E_{xx}=\Id_{M_x}=F_{xx}$.
\end{proof}

A tree factorization will transfer endpoint dependence to homotopies
whose images have zero two-dimensional Hausdorff content. We first
check that projection preserves the sewn transport.

\section{Factorization and descent}
\label{sec:descent}

The relevant tree factorization has a $1$-Lipschitz projection. We first
verify that pullback preserves the defect and chain bounds and that
sewing commutes with this projection.

\begin{proposition}\label{prop:pullback}
Let $q:T\to P$ be $1$-Lipschitz, where $T$ is any metric space.
Set
\[
 N_u=M_{q(u)},\qquad \nu_{uv}=\mu_{q(u)q(v)}:N_v\to N_u.
\]
Then $\nu$ is an approximate action with the same
$C_r,a_r,b_r,p,g$ and with
$f^\uparrow(r)=\max_{0\le s\le r}f(s)$ in place of $f$.
For every Lipschitz path $\alpha$ in $T$,
\begin{equation}\label{eq:naturality}
 \Phi_\nu(\alpha)=\Phi_\mu(q\circ\alpha).
\end{equation}
\end{proposition}
\begin{proof}
The fibers $N_u$ are complete and
$\nu_{uu}=\mu_{q(u)q(u)}=\Id_{N_u}$.
The maximum defining $f^\uparrow$ exists on each compact interval.
It is nondecreasing and $f^\uparrow(0)=0$.
For $R>0$, let
\[
 \omega_R(h)=\sup\{|f(v)-f(w)|:
          v,w\in[0,R],\ |v-w|\le h\}.
\]
Uniform continuity gives $\omega_R(h)\to0$ as $h\downarrow0$.
If $0\le s\le t\le R$, a point attaining the maximum on $[0,t]$
either belongs to $[0,s]$ or belongs to $(s,t]$. In the two cases,
respectively,
\[
 f^\uparrow(t)=f^\uparrow(s)
 \quad\text{or}\quad
 0\le f^\uparrow(t)-f^\uparrow(s)\le\omega_R(t-s).
\]
Thus $f^\uparrow$ is continuous, and
\[
 \Lip(\nu_{uv})\le1+f(d_P(q(u),q(v)))
                   \le1+f^\uparrow(d_T(u,v)).
\]
For every $u,v,w\in T$,
\[
 \begin{aligned}
 \dist_{uv}(\nu_{uv},\nu_{uw}\nu_{wv})
 &\le\sum_{r=1}^N C_r
       d_P(q(v),q(w))^{a_r}d_P(q(w),q(u))^{b_r}\\
 &\le\sum_{r=1}^N C_r d_T(v,w)^{a_r}d_T(w,u)^{b_r}.
 \end{aligned}
\]
For a finite chain $u_0,\ldots,u_n$,
\[
 \begin{aligned}
 \Lip(\nu_{u_0u_1}\cdots\nu_{u_{n-1}u_n})
 &\le g\left(\sum_{j=1}^n d_P(q(u_{j-1}),q(u_j))\right)\\
 &\le g\left(\sum_{j=1}^n d_T(u_{j-1},u_j)\right).
 \end{aligned}
\]
These are the required approximate-action inequalities.
For every partition $D$, the identity of individual factors gives
\[
 P_\nu(\alpha,D)=P_\mu(q\circ\alpha,D).
\]
Their limits lie in the same extended supremum map space, so
Proposition~\ref{prop:path} gives \eqref{eq:naturality}.
In particular, $q\circ\alpha=c_a$ implies
$\Phi_\nu(\alpha)=\Id_{M_a}$.
\end{proof}

\begin{corollary}\label{cor:common-tree}
Suppose $q:T\to P$ is $1$-Lipschitz, $T$ is a metric tree, and
$\alpha_0,\alpha_1:[0,1]\to T$ are Lipschitz paths with the same
endpoints. Then
\[
 \Phi_\mu(q\circ\alpha_0)=\Phi_\mu(q\circ\alpha_1).
\]
\end{corollary}
\begin{proof}
For the pulled-back action of Proposition~\ref{prop:pullback},
Theorem~\ref{thm:tree} and \eqref{eq:naturality} give
\[
 \Phi_\mu(q\alpha_0)=\Phi_\nu(\alpha_0)
 =F_{\alpha_0(0)\alpha_0(1)}
 =F_{\alpha_1(0)\alpha_1(1)}
 =\Phi_\nu(\alpha_1)=\Phi_\mu(q\alpha_1).
\]
\end{proof}

\begin{proposition}\label{prop:tree-square}
Let $H:[0,1]^2\to P$ factor as $H=q\psi$, where
$\psi:[0,1]^2\to T$ is Lipschitz, $T$ is a metric tree, and
$q:T\to P$ is $1$-Lipschitz. Define the boundary paths
\[
 \gamma_i(s)=H(s,i),\qquad \beta_i(t)=H(i,t)
 \quad(i=0,1).
\]
Then
\begin{equation}\label{eq:tree-boundary}
 \Phi_\mu(\gamma_0)\Phi_\mu(\beta_1)
 =\Phi_\mu(\beta_0)\Phi_\mu(\gamma_1).
\end{equation}
If $H$ is relative, then
$\Phi_\mu(\gamma_0)=\Phi_\mu(\gamma_1)$.
\end{proposition}
\begin{proof}
Put $\alpha_i(s)=\psi(s,i)$ and $\eta_i(t)=\psi(i,t)$.
The Lipschitz paths $\alpha_0*\eta_1$ and
$\eta_0*\alpha_1$ both run from $\psi(0,0)$ to $\psi(1,1)$.
Corollary~\ref{cor:common-tree} and concatenation give
\[
 \begin{aligned}
 \Phi_\mu(\gamma_0)\Phi_\mu(\beta_1)
 &=\Phi_\mu\bigl(q\circ(\alpha_0*\eta_1)\bigr)\\
 &=\Phi_\mu\bigl(q\circ(\eta_0*\alpha_1)\bigr)
  =\Phi_\mu(\beta_0)\Phi_\mu(\gamma_1).
 \end{aligned}
\]
If $H$ is relative from $a$ to $b$, then
$q\eta_0=\beta_0=c_a$ and $q\eta_1=\beta_1=c_b$.
Consequently the vertical transports are $\Id_{M_a}$ and
$\Id_{M_b}$, and \eqref{eq:tree-boundary} gives the assertion.
The lifted vertical paths $\eta_0,\eta_1$ need not be constant.
\end{proof}

For a subset $E$ of a metric space, use the unnormalized
two-dimensional Hausdorff content
\[
 \HH^2_\infty(E)=\inf\left\{
       \sum_i(\diam U_i)^2:E\subseteq\bigcup_iU_i\right\},
\]
where covers are countable. A fixed positive dimensional
normalization does not change its zero sets.
For a Lipschitz map $H:Q=[0,1]^2\to P$, the $(2,0)$-mapping
content is
\[
 \HH^{2,0}_\infty(H,Q)
 =\inf_{\mathcal Q}\sum_{R\in\mathcal Q}
             \HH^2_\infty(H(R)),
\]
where $\mathcal Q$ ranges over covers of $Q$ by closed dyadic
subsquares admissible under the conventions of \cite{EH}.
The cover $\mathcal Q=\{Q\}$
is admissible. For every admissible cover, countable subadditivity
also gives
\[
 \HH^2_\infty(H(Q))
 \le\sum_{R\in\mathcal Q}\HH^2_\infty(H(R)).
\]
Taking the infimum and then using the one-square cover proves
\begin{equation}\label{eq:content}
 \HH^{2,0}_\infty(H,Q)=\HH^2_\infty(H(Q)).
\end{equation}

The geometric input is the implication (e)$\Rightarrow$(a) of
\cite[Theorem~1.1]{EH}, together with \cite[Remark~1.2]{EH}.
For a Lipschitz map of the Euclidean square into an arbitrary
metric space, it gives
\begin{equation}\label{eq:factorization}
 \begin{gathered}
 \HH^{2,0}_\infty(H,Q)=0
 \quad\Longrightarrow\quad
 H=q\circ\psi,\\
 \psi:Q\to T,\quad q:T\to P,\quad T\text{ a metric tree},\\
 \Lip(\psi)\le\Lip(H),\qquad \Lip(q)\le1.
 \end{gathered}
\end{equation}
The square in this application carries its Euclidean metric.
Changing between the Euclidean and $\ell^1$ metrics preserves its
Lipschitz maps and does not alter the image-content condition.
No completeness or geodesic hypothesis on $P$ is required.

\begin{theorem}\label{thm:descent}
Let $H:Q\to P$ be a relative Lipschitz homotopy from
$\gamma_0:a\to b$ to $\gamma_1:a\to b$. If
$\HH^2_\infty(H(Q))=0$, then
\[
 \Phi_\mu(\gamma_0)=\Phi_\mu(\gamma_1).
\]
\end{theorem}
\begin{proof}
Equation~\eqref{eq:content} gives
$\HH^{2,0}_\infty(H,Q)=0$.
The factorization \eqref{eq:factorization} satisfies all the
hypotheses of Proposition~\ref{prop:tree-square}.
Its relative-boundary conclusion proves the equality.
\end{proof}

\begin{corollary}\label{cor:thin}
Every Lipschitz-thin loop $\lambda:a\to a$ satisfies
$\Phi_\mu(\lambda)=\Id_{M_a}$.
Thin-equivalent Lipschitz paths have equal sewn transports.
\end{corollary}
\begin{proof}
For any Lipschitz path $\lambda$ and integer $n\ge1$, the sets
$U_j=\lambda([(j-1)/n,j/n])$ cover its image and satisfy
$\diam U_j\le\Lip(\lambda)/n$. Hence
\[
 0\le\HH^2_\infty(\lambda([0,1]))
 \le\sum_{j=1}^n(\diam U_j)^2
 \le\frac{\Lip(\lambda)^2}{n}.
\]
Letting $n\to\infty$ proves that this content is zero.
If $H$ is a thin contraction of $\lambda$, then
\[
 H(Q)\subseteq\lambda([0,1]),\qquad
 \HH^2_\infty(H(Q))=0.
\]
Theorem~\ref{thm:descent} therefore gives
$\Phi_\mu(\lambda)=\Phi_\mu(c_a)=\Id_{M_a}$.
For a finite thin chain $\beta_0,\ldots,\beta_m$, each loop
$\beta_{j-1}*\bar\beta_j$ is thin. Proposition~\ref{prop:path-algebra}
and the loop conclusion imply
\[
 \Id=\Phi_\mu(\beta_{j-1}*\bar\beta_j)
 =\Phi_\mu(\beta_{j-1})\Phi_\mu(\beta_j)^{-1},
 \qquad \Phi_\mu(\beta_{j-1})=\Phi_\mu(\beta_j).
\]
Transitivity gives
$\Phi_\mu(\beta_0)=\Phi_\mu(\beta_m)$.
\end{proof}

Transport is constant on thin classes. We can now assemble the groupoid
action, prove its uniqueness, and express its bounds through the size
of a class.

\section{Sewing on the Lipschitz-thin groupoid}
\label{sec:main}

We first compare the infimal length and the infimal Lipschitz constant
within a thin class. Their equality puts the global estimate in the
normalization used in \cite[Conjecture~4.13]{CM}.

\begin{lemma}[Length representatives]\label{lem:class-length}
For every nonconstant Lipschitz path $\gamma$ of length $\ell$, there
is a Lipschitz path $\widehat\gamma$ in the same thin class such that
\[
 \Lip(\widehat\gamma)=\len(\widehat\gamma)=\ell.
\]
Consequently, for every thin class $h$,
\begin{equation}\label{eq:class-length}
 \mathcal L(h):=\inf_{\gamma\in h}\len(\gamma)
       =\inf_{\gamma\in h}\Lip(\gamma),\qquad
 \mathcal L(h^{-1})=\mathcal L(h).
\end{equation}
For composable classes, $\mathcal L(hk)\le\mathcal L(h)+\mathcal L(k)$.
\end{lemma}
\begin{proof}
For $V(t)=\len(\gamma|_{[0,t]})$, Lemma~\ref{lem:length-control}
gives a nondecreasing continuous surjection $V:[0,1]\to[0,\ell]$.
If $V(s)=V(t)$, then
$d_P(\gamma(s),\gamma(t))\le|V(s)-V(t)|=0$.
Thus the prescription
\[
 \beta(v)=\gamma(t)\quad\text{whenever }V(t)=v
\]
defines a map $\beta:[0,\ell]\to P$. For $v,w$ and corresponding
preimages $s,t$, Lemma~\ref{lem:length-control} gives
\[
 d_P(\beta(v),\beta(w))
 =d_P(\gamma(s),\gamma(t))\le|V(s)-V(t)|=|v-w|.
\]
Set $\widehat\gamma(u)=\beta(\ell u)$ and $\rho(t)=V(t)/\ell$.
Then
\[
 \Lip(\widehat\gamma)\le\ell,\qquad
 \Lip(\rho)\le\Lip(\gamma)/\ell,\qquad
 \rho(0)=0,\quad\rho(1)=1,\quad
 \gamma=\widehat\gamma\circ\rho.
\]
Lemma~\ref{lem:thin-calculus} gives $[\gamma]=[\widehat\gamma]$.
For every partition $D$, the list $\rho(D)$ is nondecreasing, and
deleting repetitions gives
$S_\gamma(D)\le\len(\widehat\gamma)$. Therefore
\[
 \ell\le\len(\widehat\gamma)
       \le\Lip(\widehat\gamma)\le\ell.
\]
For constant paths the same equalities hold with length zero.
Applying this construction to each representative, and using
$\len(\gamma)\le\Lip(\gamma)$ in the other direction, proves
the equality of infima in \eqref{eq:class-length}. Reversal gives
a length-preserving bijection between the representatives of $h$
and those of $h^{-1}$.

Length additivity and the affine definition of concatenation give
$\len(\gamma*\eta)=\len(\gamma)+\len(\eta)$. For every
$\delta>0$, representatives with lengths below
$\mathcal L(h)+\delta$ and $\mathcal L(k)+\delta$ therefore give
$\mathcal L(hk)\le\mathcal L(h)+\mathcal L(k)+2\delta$.
Let $\delta\downarrow0$.
\end{proof}

\begin{theorem}[Thin-groupoid sewing]\label{thm:main}
Every approximate action in Definition~\ref{def:approx} determines
an action $\varphi$ of $\Gthin$ on $(M_x)$ by bi-Lipschitz maps.
For every representative $\gamma:a\to b$,
\begin{equation}\label{eq:main-limit}
 \varphi_{[\gamma]}=\Phi_\mu(\gamma)
 =\lim_{\mesh(D)\to0}P_\gamma(D)\in\mathcal C_{ab}.
\end{equation}
The limit is independent of the representative and also holds when
the length mesh $\omega_\gamma(D)$ tends to zero. For all $s,t$,
\begin{equation}\label{eq:main-local}
 \begin{split}
 \dist_{\gamma(s)\gamma(t)}
  (\varphi_{[\gamma_{s,t}]},\mu_{\gamma(s)\gamma(t)})
 &\le K g(\ell_{s,t})\ell_{s,t}^{p}\\
 &\le K g(\ell_{s,t})\Lip(\gamma)^p|t-s|^p.
 \end{split}
\end{equation}
For every class $h$,
\begin{equation}\label{eq:main-lip}
 \max\{\Lip(\varphi_h),\Lip(\varphi_h^{-1})\}
 \le g(\mathcal L(h))
 =g\left(\inf_{\eta\in h}\Lip(\eta)\right).
\end{equation}
The action is unique among actions $\Psi$ such that, for every
Lipschitz path $\gamma$, there are $A_\gamma<\infty$ and
$q_\gamma>1$ with
\begin{equation}\label{eq:unique-hyp}
 \dist_{\gamma(s)\gamma(t)}
 (\Psi_{[\gamma_{s,t}]},\mu_{\gamma(s)\gamma(t)})
 \le A_\gamma|t-s|^{q_\gamma}\qquad(s,t\in[0,1]).
\end{equation}
In particular, it is the unique action satisfying
\eqref{eq:main-local}.
\end{theorem}
\begin{proof}
For a link $\beta_{j-1}*\bar\beta_j$ in a thin chain,
Corollary~\ref{cor:thin} and \eqref{eq:path-algebra} give
\[
 \Id=\Phi_\mu(\beta_{j-1}*\bar\beta_j)
       =\Phi_\mu(\beta_{j-1})\Phi_\mu(\beta_j)^{-1}.
\]
Multiplication on the right by $\Phi_\mu(\beta_j)$ proves equality
of the two transports. Induction over the finite chain establishes
representative independence in \eqref{eq:main-limit}.
The path identities now give
\begin{equation}\label{eq:main-action}
 \varphi_{[\gamma][\eta]}
   =\Phi_\mu(\gamma*\eta)
   =\varphi_{[\gamma]}\varphi_{[\eta]},\qquad
 \varphi_{1_a}=\Id_{M_a},\qquad
 \varphi_{h^{-1}}=\varphi_h^{-1}.
\end{equation}
Hence these maps form an action on the groupoid of
Lemma~\ref{lem:quotient}. The affine identity in
Proposition~\ref{prop:path-algebra} and \eqref{eq:local} prove
\eqref{eq:main-local}. The two mesh convergence assertions are
those of Lemma~\ref{lem:mesh-limit}.

For every $\eta\in h$, \eqref{eq:path-lip} gives
$\Lip(\varphi_h)\le g(\len\eta)$. If $a=\mathcal L(h)$,
the defining infimum has representatives of length less than
$a+1/n$ for every integer $n\ge1$. Thus
\[
 \Lip(\varphi_h)\le g(a+1/n)\longrightarrow g(a).
\]
Continuity of $g$ and Lemma~\ref{lem:class-length} yield
\eqref{eq:main-lip}, including the bound for the inverse class.
Both maps are therefore Lipschitz and continuous on all galaxies.

For uniqueness, fix $\gamma$ and $D=(0=t_0<\cdots<t_n=1)$.
By Corollary~\ref{cor:subpath-class} and the action property,
\[
 \Psi_{[\gamma]}
 =\Psi_{[\gamma_{t_0,t_1}]}\cdots
       \Psi_{[\gamma_{t_{n-1},t_n}]}.
\]
Apply \eqref{eq:hybrid} with the original $\mu$ factors in the
left prefixes and the competing maps in the right suffixes. The
prefixes have chain length at most $\len\gamma$, so
\[
 \begin{aligned}
 \dist_{ab}(P_\gamma(D),\Psi_{[\gamma]})
 &\le g(\len\gamma)A_\gamma
              \sum_{j=1}^n(t_j-t_{j-1})^{q_\gamma}\\
 &\le g(\len\gamma)A_\gamma
              \mesh(D)^{q_\gamma-1}\longrightarrow0.
 \end{aligned}
\]
The same products converge to $\varphi_{[\gamma]}$ by
\eqref{eq:main-limit}. Separation gives
$\Psi_{[\gamma]}=\varphi_{[\gamma]}$ for every class. No
Lipschitz assumption on the competing maps enters this comparison.
The constructed action satisfies \eqref{eq:unique-hyp} with
$q_\gamma=p$ and
$A_\gamma=Kg(\len\gamma)\Lip(\gamma)^p$.
\end{proof}

\begin{proposition}[The path-scaling normalization]\label{rem:scaling}
The constants of the approximate flow pulled back along $\gamma$
can be taken to be
\begin{equation}\label{eq:scaled-constants}
 C_r^\gamma=C_r\Lip(\gamma)^p.
\end{equation}
In terms of these constants, the local comparison in
Theorem~\ref{thm:main} implies
\begin{equation}\label{eq:CM-corrected}
 \dist(\varphi_{[\gamma_{s,t}]},\mu_{\gamma(s)\gamma(t)})
 \le 2^p g(\Lip(\gamma)|t-s|)
        \left(\sum_r C_r^\gamma\right)\zeta(p)|t-s|^p.
\end{equation}
The factor $\Lip(\gamma)^p$ cannot be omitted from a uniform
estimate expressed in terms of the original $C_r$.
\end{proposition}
\begin{proof}
Substitution of $x=\gamma(s)$, $z=\gamma(u)$, $y=\gamma(t)$
in \eqref{eq:defect} gives
\[
 \dist(\mu_{\gamma(s)\gamma(t)},
       \mu_{\gamma(s)\gamma(u)}\mu_{\gamma(u)\gamma(t)})
 \le\sum_r C_r\Lip(\gamma)^{a_r+b_r}
                         |t-u|^{a_r}|u-s|^{b_r}.
\]
This proves \eqref{eq:scaled-constants}. Equation~\eqref{eq:CM-corrected}
uses $a_r+b_r=p$, \eqref{eq:main-local}, and
$\ell_{s,t}\le\Lip(\gamma)|t-s|$.

For necessity let $P=\R$, $M_x=\R$, and
$\mu_{xy}(v)=v+(y-x)^2$. All maps are translations, so $f=0$,
$g=1$, and
\[
 \mu_{xx}=\Id,\qquad
 \dist(\mu_{xy},\mu_{xz}\mu_{zy})=2|(y-z)(z-x)|.
\]
Thus $p=2$, $N=1$, $C_1=2$, and $a_1=b_1=1$ are admissible.
For $\gamma_L(t)=Lt$, $L>0$, the partition product translates by
\[
 L^2\sum_j(t_j-t_{j-1})^2
 \le L^2\mesh(D)\longrightarrow0.
\]
Hence $\varphi_{[\gamma_L]}=\Id$ and
\[
 \dist(\varphi_{[\gamma_L]},\mu_{0L})=L^2.
\]
The expression in \cite[Conjecture~4.13]{CM} without the scaling
factor is $8\zeta(2)$ for these fixed approximate-action data.
It fails whenever $L^2>8\zeta(2)$.
\end{proof}

Theorem~\ref{thm:main} therefore proves the full structural assertion
of Conjecture~4.13 and its correctly scaled local estimate. The
equivalence relation, fiber class, and three-point hypotheses are
unchanged. The length estimate in \eqref{eq:main-local} is intrinsic
to the represented subpath, and \eqref{eq:CM-corrected} recovers the
parameter-based formulation.

The action also assigns transport to based loops. We next construct its
holonomy representation and identify the condition for factorization
through relative Lipschitz homotopy.

\section{Based holonomy and homotopy factorization}
\label{holo:sec}

For the holonomy anticipated in \cite[Remark~4.15]{CM}, the domain at
$a$ is the isotropy group of the thin groupoid at $a$. We first define
the relative Lipschitz-homotopy quotient, keeping the source and
multiplication conventions of Section~\ref{sec:paths-quotient}.

\begin{proposition}\label{holo:quotient}
Relative Lipschitz homotopy defines a groupoid
$\mathcal G_P^{\mathrm{Lip}}$ whose objects are the points of $P$.
For a path $\gamma:a\to b$, its class $[\gamma]_{\mathrm{Lip}}$
has source $b$ and range $a$. The operations are
\[
 [\gamma]_{\mathrm{Lip}}[\eta]_{\mathrm{Lip}}
    =[\gamma*\eta]_{\mathrm{Lip}},\qquad
 1_a=[c_a]_{\mathrm{Lip}},\qquad
 [\gamma]_{\mathrm{Lip}}^{-1}=[\bar\gamma]_{\mathrm{Lip}}.
\]
The map
\begin{equation}\label{holo:q}
 q:\Gthin\longrightarrow\mathcal G_P^{\mathrm{Lip}},\qquad
 q([\gamma])=[\gamma]_{\mathrm{Lip}},
\end{equation}
is a surjective groupoid functor and is the identity on objects.
\end{proposition}
\begin{proof}
Reflexivity is given by $H(s,t)=\gamma(s)$. If $H$ joins
$\gamma_0$ to $\gamma_1$, then $H(s,1-t)$ joins $\gamma_1$ to
$\gamma_0$. If $H^{01}$ joins $\gamma_0$ to $\gamma_1$ and
$H^{12}$ joins $\gamma_1$ to $\gamma_2$, then
\[
 H^{02}(s,t)=
 \begin{cases}
 H^{01}(s,2t),&0\le t\le\tfrac12,\\
 H^{12}(s,2t-1),&\tfrac12\le t\le1
 \end{cases}
\]
joins $\gamma_0$ to $\gamma_2$. Both branches agree at $t=1/2$,
and finite Lipschitz gluing proves the required regularity.
For homotopies $H^\gamma$ and $H^\eta$ with fixed common junction
$b$, the formula
\[
 J(s,t)=
 \begin{cases}
 H^\gamma(2s,t),&0\le s\le\tfrac12,\\
 H^\eta(2s-1,t),&\tfrac12\le s\le1
 \end{cases}
\]
is a relative Lipschitz homotopy of their concatenations. The
branches agree at $s=1/2$ because both values are $b$.
Consequently multiplication is well defined on homotopy classes.

For any endpoint-preserving Lipschitz map $\rho:[0,1]\to[0,1]$,
\[
 R(s,t)=\gamma((1-t)s+t\rho(s))
\]
is a relative Lipschitz homotopy from $\gamma$ to $\gamma\circ\rho$.
The unit paths insert constant pauses. The two bracketings of a
threefold concatenation are increasing piecewise affine
reparameterizations of the path that assigns one third of the
parameter interval to each factor. Their classes therefore agree.
Finally,
\[
 S_\gamma(s,t)=\gamma((1-t)\min\{2s,2-2s\})
\]
contracts $\gamma*\bar\gamma$ relative to its basepoint. Applying
the formula to $\bar\gamma$ contracts the reverse product.
This proves the groupoid identities.

Let $\beta_0,\ldots,\beta_m$ be a thin chain. Each elementary loop
has a relative Lipschitz contraction, so
\[
 [\beta_{j-1}]_{\mathrm{Lip}}
 [\beta_j]_{\mathrm{Lip}}^{-1}
 =[\beta_{j-1}*\bar\beta_j]_{\mathrm{Lip}}=1.
\]
Right multiplication gives
$[\beta_{j-1}]_{\mathrm{Lip}}=[\beta_j]_{\mathrm{Lip}}$.
Thus \eqref{holo:q} is well defined. The displayed operations use
the same representative paths in both groupoids, so $q$ preserves
them. Every morphism of $\mathcal G_P^{\mathrm{Lip}}$ is the image
of the thin class of one of its representatives.
\end{proof}

For each $a\in P$, write
\begin{equation}\label{holo:groups}
 \pi_{1,\mathrm{thin}}^{\mathrm{Lip}}(P,a)
       =\operatorname{Aut}_{\Gthin}(a),\qquad
 \pi_1^{\mathrm{Lip}}(P,a)
       =\operatorname{Aut}_{\mathcal G_P^{\mathrm{Lip}}}(a).
\end{equation}
Restriction of $q$ gives a surjective group homomorphism
\[
 q_a:\pi_{1,\mathrm{thin}}^{\mathrm{Lip}}(P,a)
            \longrightarrow\pi_1^{\mathrm{Lip}}(P,a).
\]
Its kernel consists exactly of thin classes represented by based
loops admitting a relative Lipschitz contraction. This property
holds for every representative of such a class, because $q$ is
well defined.

For a bijection $F:M_a\to M_a$, define its displacement by
\[
 \operatorname{disp}_a(F)=\sup_{u\in M_a}d_a(Fu,u).
\]
Let $\operatorname{BiLip}_{\mathrm{bd}}(M_a)$ be the set of
bijections for which $F,F^{-1}$ are Lipschitz and
$\operatorname{disp}_a(F)<\infty$.

\begin{lemma}\label{holo:target}
The set $\operatorname{BiLip}_{\mathrm{bd}}(M_a)$ is a group under
composition. Its elements preserve every galaxy of $M_a$ setwise.
For $F,G$ in this group,
\begin{equation}\label{holo:disp}
 \operatorname{disp}_a(FG)
 \le\operatorname{disp}_a(F)+\operatorname{disp}_a(G),\qquad
 \operatorname{disp}_a(F^{-1})=\operatorname{disp}_a(F).
\end{equation}
\end{lemma}
\begin{proof}
The identity has zero displacement. Products and their inverses
are Lipschitz, with
\[
 \Lip(FG)\le\Lip(F)\Lip(G),\qquad
 \Lip((FG)^{-1})\le\Lip(G^{-1})\Lip(F^{-1}).
\]
For each $u\in M_a$,
\[
 d_a(FG(u),u)
 \le d_a(FG(u),G(u))+d_a(G(u),u)
 \le\operatorname{disp}_a(F)+\operatorname{disp}_a(G).
\]
Surjectivity permits the substitution $v=F(u)$ and gives
\[
 \sup_{v\in M_a}d_a(F^{-1}(v),v)
   =\sup_{u\in M_a}d_a(u,F(u)).
\]
These estimates prove \eqref{holo:disp} and closure under products
and inverses. The finite bound $d_a(Fu,u)\le\operatorname{disp}_a(F)$
places $Fu$ in the galaxy of $u$. The same argument for $F^{-1}$
gives preservation of that galaxy onto itself.
\end{proof}

\begin{theorem}\label{holo:representation}
For every $a\in P$, the sewn action defines a homomorphism
\begin{equation}\label{holo:hol}
 \operatorname{Hol}_a:
 \pi_{1,\mathrm{thin}}^{\mathrm{Lip}}(P,a)
 \longrightarrow\operatorname{BiLip}_{\mathrm{bd}}(M_a),\qquad
 \operatorname{Hol}_a([\lambda])=\Phi_\mu(\lambda).
\end{equation}
For every representative $\lambda$ of $h$,
\begin{equation}\label{holo:limit}
 \operatorname{Hol}_a(h)
 =\lim_{\mesh(D)\to0}
       \mu_{\lambda(t_0)\lambda(t_1)}\cdots
       \mu_{\lambda(t_{n-1})\lambda(t_n)}.
\end{equation}
Put $L(h)=\inf_{\lambda\in h}\len(\lambda)$. Then
\begin{align}
 \max\{\Lip(\operatorname{Hol}_a(h)),
          \Lip(\operatorname{Hol}_a(h)^{-1})\}
 &\le g(L(h)),\label{holo:bound}\\
 \operatorname{disp}_a(\operatorname{Hol}_a(h))
 &\le K g(L(h))L(h)^p<\infty.
       \label{holo:displacement}
\end{align}
This homomorphism is determined uniquely as the restriction of
an action satisfying the local comparison property
\eqref{eq:unique-hyp}.
\end{theorem}
\begin{proof}
Theorem~\ref{thm:main} gives representative independence and
\[
 \operatorname{Hol}_a(hk)
  =\operatorname{Hol}_a(h)\operatorname{Hol}_a(k),\qquad
 \operatorname{Hol}_a(1)=\Id,\qquad
 \operatorname{Hol}_a(h^{-1})=\operatorname{Hol}_a(h)^{-1}.
\]
For a loop $\lambda:a\to a$, the intrinsic local estimate in
Theorem~\ref{thm:main} at $s=0,t=1$ and $\mu_{aa}=\Id$ give
\[
 \operatorname{disp}_a(\Phi_\mu(\lambda))
 \le K g(\len\lambda)\len(\lambda)^p.
\]
The right-hand side is finite. Lemma~\ref{holo:target} therefore
identifies the target group in \eqref{holo:hol}. The function
$r\mapsto g(r)r^p$ is continuous and nondecreasing. Taking the
infimum over all representatives proves \eqref{holo:displacement}.
The partition limit is \eqref{eq:main-limit}. Reversal is a
bijection from the representatives of $h$ to those of $h^{-1}$
and preserves length. Thus $L(h^{-1})=L(h)$, and applying
\eqref{eq:main-lip} to both classes proves \eqref{holo:bound}.
The last assertion is the restriction of the uniqueness statement
in Theorem~\ref{thm:main}. It imposes no uniqueness claim on
arbitrary homomorphisms of the isotropy group.
\end{proof}

\begin{proposition}\label{holo:basepoint}
Let $\sigma:a\to b$ be a Lipschitz path and put
$Q=\Phi_\mu(\sigma):M_b\to M_a$. The map
\[
 C_\sigma:\pi_{1,\mathrm{thin}}^{\mathrm{Lip}}(P,b)
       \longrightarrow\pi_{1,\mathrm{thin}}^{\mathrm{Lip}}(P,a),
 \qquad C_\sigma(h)=[\sigma]h[\sigma]^{-1},
\]
is an isomorphism depending only on $[\sigma]$, and
\begin{equation}\label{holo:covariance}
 \operatorname{Hol}_a(C_\sigma(h))
       =Q\operatorname{Hol}_b(h)Q^{-1}.
\end{equation}
Conjugation by $Q$ maps
$\operatorname{BiLip}_{\mathrm{bd}}(M_b)$ isomorphically onto
$\operatorname{BiLip}_{\mathrm{bd}}(M_a)$.
\end{proposition}
\begin{proof}
The groupoid identities give
\[
 C_\sigma(h)C_\sigma(k)
  =[\sigma]h[\sigma]^{-1}[\sigma]k[\sigma]^{-1}
  =[\sigma]hk[\sigma]^{-1},\qquad
 C_{\bar\sigma}C_\sigma=\Id,
 \quad C_\sigma C_{\bar\sigma}=\Id.
\]
The action law proves \eqref{holo:covariance}. Its factors have
the types
\[
 M_a\xrightarrow{Q^{-1}}M_b
 \xrightarrow{\operatorname{Hol}_b(h)}M_b
 \xrightarrow{Q}M_a.
\]
For $F\in\operatorname{BiLip}_{\mathrm{bd}}(M_b)$, substitution
$v=Q(u)$ gives
\[
 \operatorname{disp}_a(QFQ^{-1})
   =\sup_{u\in M_b}d_a(QF(u),Q(u))
   \le\Lip(Q)\operatorname{disp}_b(F)<\infty.
\]
Both $QFQ^{-1}$ and its inverse are Lipschitz. Applying the
same estimate to $Q^{-1}$ proves that conjugation is onto and
has the stated inverse.
\end{proof}

\begin{theorem}\label{holo:flatness}
For the sewn action $\varphi$, the following conditions are
equivalent.
\begin{enumerate}[label=\textup{(\roman*)}]
\item There is an action $\widehat\varphi$ of
$\mathcal G_P^{\mathrm{Lip}}$ on the given fibers satisfying
$\varphi_h=\widehat\varphi_{q(h)}$ for every $h\in\Gthin$.
\item Relatively Lipschitz-homotopic paths have equal sewn transport.
\item Every based loop with a relative Lipschitz contraction has
identity sewn transport.
\item For every $a\in P$,
$\ker q_a\subseteq\ker\operatorname{Hol}_a$.
\item For every $a\in P$ there is a homomorphism
\[
 \widehat{\operatorname{Hol}}_a:
 \pi_1^{\mathrm{Lip}}(P,a)
       \longrightarrow\operatorname{BiLip}_{\mathrm{bd}}(M_a),
 \qquad
 \operatorname{Hol}_a=\widehat{\operatorname{Hol}}_a q_a.
\]
\end{enumerate}
Whenever they exist, the factorizations in \textup{(i)} and
\textup{(v)} are unique. Define
\begin{equation}\label{holo:obstruction}
 \mathcal H_a=\operatorname{im}\operatorname{Hol}_a,
 \qquad
 \mathcal N_a=\operatorname{Hol}_a(\ker q_a).
\end{equation}
Then $\mathcal N_a$ is a normal subgroup of $\mathcal H_a$.
Conditions \textup{(i)--(v)} hold exactly when
$\mathcal N_a=\{\Id_{M_a}\}$ for every $a$.
\end{theorem}
\begin{proof}
If (i) holds and $[\gamma_0]_{\mathrm{Lip}}
=[\gamma_1]_{\mathrm{Lip}}$, then
\[
 \Phi_\mu(\gamma_0)
 =\widehat\varphi_{[\gamma_0]_{\mathrm{Lip}}}
 =\widehat\varphi_{[\gamma_1]_{\mathrm{Lip}}}
 =\Phi_\mu(\gamma_1).
\]
Thus (ii) holds. Conversely, (ii) makes
$\widehat\varphi_{[\gamma]_{\mathrm{Lip}}}=\Phi_\mu(\gamma)$
well defined. Concatenation, constant paths and reversal give
the action identities, proving (i).

Condition (ii) applied to a loop and its constant contraction
endpoint gives (iii). Under (iii), let
$[\gamma_0]_{\mathrm{Lip}}=[\gamma_1]_{\mathrm{Lip}}$.
Proposition~\ref{holo:quotient} gives
\[
 [\gamma_0*\bar\gamma_1]_{\mathrm{Lip}}
   =[\gamma_0]_{\mathrm{Lip}}
    [\gamma_1]_{\mathrm{Lip}}^{-1}=1.
\]
The represented loop therefore admits a relative Lipschitz
contraction. Hence
\[
 \Id=\Phi_\mu(\gamma_0*\bar\gamma_1)
     =\Phi_\mu(\gamma_0)\Phi_\mu(\gamma_1)^{-1},
\]
which gives (ii) after right composition by
$\Phi_\mu(\gamma_1)$. The description of $\ker q_a$ after
\eqref{holo:groups} proves equivalence of (iii) and (iv).

Assume (iv). For $\widehat h\in\pi_1^{\mathrm{Lip}}(P,a)$,
define
\[
 \widehat{\operatorname{Hol}}_a(\widehat h)
      =\operatorname{Hol}_a(h)
      \quad\text{for any }h\text{ with }q_a(h)=\widehat h.
\]
Such preimages exist by surjectivity. If $q_a(h')=q_a(h)$, then
$h'^{-1}h\in\ker q_a$, so
\[
 \operatorname{Hol}_a(h')^{-1}\operatorname{Hol}_a(h)
 =\operatorname{Hol}_a(h'^{-1}h)=\Id.
\]
The value is independent of the preimage. If $h,k$ represent
$\widehat h,\widehat k$ under $q_a$, then $hk$ represents their
product, and
\[
 \widehat{\operatorname{Hol}}_a(\widehat h\widehat k)
 =\operatorname{Hol}_a(hk)
 =\widehat{\operatorname{Hol}}_a(\widehat h)
  \widehat{\operatorname{Hol}}_a(\widehat k).
\]
This proves (v). Conversely, (v) gives
$\operatorname{Hol}_a(h)
=\widehat{\operatorname{Hol}}_a(1)=\Id$ whenever
$h\in\ker q_a$, proving (iv). Surjectivity of $q$ and $q_a$
proves uniqueness of both factorizations.

For $g\in\pi_{1,\mathrm{thin}}^{\mathrm{Lip}}(P,a)$ and
$n\in\ker q_a$,
\[
 q_a(gng^{-1})=q_a(g)1q_a(g)^{-1}=1,
\]
and therefore
\[
 \operatorname{Hol}_a(g)\operatorname{Hol}_a(n)
 \operatorname{Hol}_a(g)^{-1}
   =\operatorname{Hol}_a(gng^{-1})\in\mathcal N_a.
\]
This proves normality. The identity
$\mathcal N_a=\{\Id\}$ is exactly the kernel inclusion in (iv).
\end{proof}

The factored action has the same partition formula for every
representative of a relative Lipschitz-homotopy class. In particular, if
$\widehat h\in\pi_1^{\mathrm{Lip}}(P,a)$ and the equivalent
conditions of Theorem~\ref{holo:flatness} hold, put
$\widehat L(\widehat h)=\inf_{\lambda\in\widehat h}\len\lambda$. Then
\begin{align}
 \max\{\Lip(\widehat{\operatorname{Hol}}_a(\widehat h)),
        \Lip(\widehat{\operatorname{Hol}}_a(\widehat h)^{-1})\}
 &\le g(\widehat L(\widehat h)),
       \label{holo:full-bound}\\
 \operatorname{disp}_a(\widehat{\operatorname{Hol}}_a(\widehat h))
 &\le K g(\widehat L(\widehat h))\widehat L(\widehat h)^p.
       \label{holo:full-disp}
\end{align}
Indeed, every representative gives the same map, and the pathwise
bounds apply to each representative and its reversal. Taking
infima and using continuity and monotonicity of $g$ proves these
inequalities.

\begin{remark}\label{holo:scope}
All contractions and homotopies in
Theorem~\ref{holo:flatness} are relative Lipschitz homotopies.
An arbitrary metric space need not permit replacement of a
continuous homotopy by a Lipschitz one. Consequently,
$\pi_1^{\mathrm{Lip}}(P,a)$ has not been identified with the
ordinary topological fundamental group. The knitting hypotheses
give the factorization criterion for the homotopies to which their
estimates apply. The basic three-point hypotheses already give
the thin representation \eqref{holo:hol}. No topology on either
loop group is needed for these algebraic homomorphisms.
\end{remark}

To verify this kernel condition from the increments, we compare boundary
transports across a Lipschitz square. A grid reduces the comparison to
a sum of local errors.

\section{Rectangular comparison and full homotopy descent}
\label{grid:section}

We retain the actual comparison error of each cell before imposing
scalar bounds. The key estimate uses only horizontal prefixes to
control how these errors accumulate. A three-point order greater than
two will then make the total error tend to zero.

Let $H:[0,1]^2\to P$ be a relative Lipschitz homotopy between
$\gamma_0,\gamma_1:a\to b$. The Euclidean metric is used on the square.
Write
\[
 L_H=\Lip(H),\qquad
 \ell_H=\sup_{0\le t\le1}\len(H(\cdot,t)),\qquad
 \Lambda_H=g(\ell_H).
\]
Every horizontal slice is $L_H$-Lipschitz on an interval of length one.
Consequently,
\begin{equation}\label{grid:slice-bound}
 0\le\ell_H\le L_H,\qquad 1\le\Lambda_H\le g(L_H)<\infty.
\end{equation}
A grid $G=(S,T)$ consists of two finite partitions
\[
 S=(0=s_0<\cdots<s_n=1),\qquad
 T=(0=t_0<\cdots<t_m=1).
\]
Its vertices and row products are
\begin{align*}
 z_{i,j}&=H(s_i,t_j),\\
 A_j(G)&=\mu_{z_{0,j}z_{1,j}}\cdots
             \mu_{z_{n-1,j}z_{n,j}}\in\mathcal C_{ab}
             \qquad(0\le j\le m).
\end{align*}
The fixed endpoints give $z_{0,j}=a$ and $z_{n,j}=b$ for every $j$.
Set
\[
 \mesh(G)=\max\{\mesh(S),\mesh(T)\}.
\]
For $1\le i\le n$ and $0\le j<m$, define
\begin{equation}\label{grid:cell-defect}
 \Box_{i,j}(G)=
 \dist_{z_{i-1,j},z_{i,j+1}}
 \left(
 \mu_{z_{i-1,j}z_{i,j}}\mu_{z_{i,j}z_{i,j+1}},
 \mu_{z_{i-1,j}z_{i-1,j+1}}\mu_{z_{i-1,j+1}z_{i,j+1}}
 \right).
\end{equation}
Both maps have source $M_{z_{i,j+1}}$ and target $M_{z_{i-1,j}}$.
The triangle inequality through $\mu_{z_{i-1,j}z_{i,j+1}}$, followed
by \eqref{eq:defect}, shows that each cell defect is finite. Denote their
total by
\[
 \Delta_H(G)=\sum_{j=0}^{m-1}\sum_{i=1}^n\Box_{i,j}(G).
\]

\begin{lemma}[Strip comparison]\label{grid:strip}
For every grid and every pair of adjacent rows,
\begin{equation}\label{grid:strip-estimate}
 \dist_{ab}(A_j(G),A_{j+1}(G))
 \le\Lambda_H\sum_{i=1}^n\Box_{i,j}(G).
\end{equation}
In particular,
\begin{equation}\label{grid:global-estimate}
 \dist_{ab}(A_0(G),A_m(G))\le\Lambda_H\Delta_H(G).
\end{equation}
\end{lemma}
\begin{proof}
Fix $j$. For $0\le i\le n$, put
\[
 x_i=z_{i,j},\qquad y_i=z_{i,j+1},\qquad v_i=\mu_{x_i y_i},
\]
and, for $1\le i\le n$, put
\[
 h_i=\mu_{x_{i-1}x_i},\qquad k_i=\mu_{y_{i-1}y_i}.
\]
Thus $v_0=\Id_{M_a}$ and $v_n=\Id_{M_b}$. The prefix and suffix
\[
 P_{i-1}=h_1\cdots h_{i-1}:M_{x_{i-1}}\to M_a,\qquad
 Q_{i+1}=k_{i+1}\cdots k_n:M_b\to M_{y_i}
\]
have the indicated types, with identities for empty products. Define
\[
 B_i=P_{i-1}h_iv_iQ_{i+1},\qquad
 C_i=P_{i-1}v_{i-1}k_iQ_{i+1}.
\]
These are maps $M_b\to M_a$, and direct multiplication gives
\[
 B_n=A_j(G),\qquad C_1=A_{j+1}(G),\qquad
 C_i=B_{i-1}\quad(2\le i\le n).
\]
Consequently,
\begin{align*}
 \dist_{ab}(A_j(G),A_{j+1}(G))
 &\le\sum_{i=1}^n\dist_{ab}(B_i,C_i),\\
 \dist_{ab}(B_i,C_i)
 &\le\Lip(P_{i-1})
   \dist_{x_{i-1},b}(h_iv_iQ_{i+1},v_{i-1}k_iQ_{i+1})\\
 &\le\Lip(P_{i-1})
   \dist_{x_{i-1},y_i}(h_iv_i,v_{i-1}k_i).
\end{align*}
The last inequality uses nonexpansiveness of right composition in the supremum
metric. Its final distance is $\Box_{i,j}(G)$. Moreover,
\[
 \Lip(P_{i-1})
 \le g\left(\sum_{r=1}^{i-1}d_P(x_{r-1},x_r)\right)
 \le g(\len(H(\cdot,t_j)))\le\Lambda_H
\]
by \eqref{eq:chain}. Substitution proves \eqref{grid:strip-estimate}.
Summation over $0\le j<m$ proves \eqref{grid:global-estimate}.
\end{proof}

The strip calculation is the finite hybrid comparison underlying the
grid argument in \cite[Theorem~4.14, Step~2]{CM}. Its useful feature
here is that every Lipschitz constant belongs to a horizontal product of
the original increments. Vertical products do not accumulate in those
constants, and no inverse estimate is used.

\begin{theorem}[A sufficient defect-sum criterion]\label{grid:descent}
Set $L_i=\Lip(\gamma_i)$ and $\ell_i=\len(\gamma_i)$ for $i=0,1$.
Every grid satisfies
\begin{equation}\label{grid:boundary-error}
 \begin{split}
 \dist_{ab}(\Phi(\gamma_0),\Phi(\gamma_1))
 \le{}&K\left(L_0^p g(\ell_0)+L_1^p g(\ell_1)\right)
                  \mesh(S)^{p-1}\\
 &+\Lambda_H\Delta_H(G).
 \end{split}
\end{equation}
In particular, let $(G_\alpha)_\alpha$ be a net of grids with
\[
 \mesh(G_\alpha)\longrightarrow0,\qquad
 \Delta_H(G_\alpha)\longrightarrow0.
\]
Then $\Phi(\gamma_0)=\Phi(\gamma_1)$. No nesting assumption on the
grids is required.
\end{theorem}
\begin{proof}
The bottom row is the partition product of $\gamma_0$ over $S$, and
the top row is that of $\gamma_1$ over $S$. The mesh estimate
\eqref{eq:rate}, on the whole parameter interval, gives
\begin{align*}
 \dist_{ab}(\Phi(\gamma_0),A_0(G))
 &\le K L_0^p g(\ell_0)\mesh(S)^{p-1},\\
 \dist_{ab}(A_m(G),\Phi(\gamma_1))
 &\le K L_1^p g(\ell_1)\mesh(S)^{p-1}.
\end{align*}
Add these inequalities to \eqref{grid:global-estimate}. Every term on
the right is finite, so the triangle inequality applies in the extended
metric and gives \eqref{grid:boundary-error}. Along the specified net
its right-hand side tends to zero. Separation in $\mathcal C_{ab}$
proves the asserted identity.
\end{proof}

The hypothesis of Theorem~\ref{grid:descent} is expressed entirely in
the original increments and a specified homotopy. It is a sufficient
estimate, not a converse characterization of descent. The following
criteria give concrete ways to verify it without first knowing the
boundary transports.

\begin{definition}\label{grid:partition-modulus}
For a nonnegative, finite-valued functional $\omega$ on nondegenerate
subintervals of $[0,1]$, define
\[
 V_\omega(\delta)=
 \sup\left\{\sum_{I\in\pi}\omega(I):
       \pi\text{ is a partition of }[0,1],\quad\mesh(\pi)\le\delta
       \right\},\qquad 0<\delta\le1.
\]
The supremum may initially be infinite. The functional is
partition-vanishing if $V_\omega(\delta)\to0$ as $\delta\downarrow0$.
\end{definition}

\begin{proposition}[Factorized cell bounds]\label{grid:factorized}
Suppose that $C_H<\infty$ and interval functionals $\omega_1,\omega_2$
satisfy, on every grid,
\begin{equation}\label{grid:factor-hyp}
 \Box_{i,j}(G)\le
 C_H\omega_1([s_{i-1},s_i])\omega_2([t_j,t_{j+1}]).
\end{equation}
Assume that their partition moduli are finite for sufficiently small
$\delta$ and that
\begin{equation}\label{grid:moduli-product}
 V_{\omega_1}(\delta)V_{\omega_2}(\delta)\longrightarrow0.
\end{equation}
Then $\Phi(\gamma_0)=\Phi(\gamma_1)$. Condition
\eqref{grid:moduli-product} holds if both functionals are
partition-vanishing. It also holds if one is partition-vanishing and
the other has bounded partition modulus near zero.
\end{proposition}
\begin{proof}
Nonnegativity and finite distributivity give
\begin{align*}
 \Delta_H(G)
 &\le C_H
  \left(\sum_{i=1}^n\omega_1([s_{i-1},s_i])\right)
  \left(\sum_{j=0}^{m-1}\omega_2([t_j,t_{j+1}])\right)\\
 &\le C_H V_{\omega_1}(\mesh(G))V_{\omega_2}(\mesh(G)).
\end{align*}
The right-hand side tends to zero on any net of grids with mesh
tending to zero. Theorem~\ref{grid:descent} applies. The two final
statements are the corresponding elementary product limits.
\end{proof}

For explicit examples, let $q\ge1$ and let
$\psi:(0,1]\to[0,\infty)$ be finite-valued and nondecreasing. The
functional $\omega(I)=|I|^q\psi(|I|)$ satisfies
\begin{equation}\label{grid:power-modulus}
 \sum_{I\in\pi}|I|^q\psi(|I|)
 \le\delta^{q-1}\psi(\delta)\sum_{I\in\pi}|I|
 =\delta^{q-1}\psi(\delta)
 \quad(\mesh(\pi)\le\delta).
\end{equation}
Thus $|I|^q$ is partition-vanishing for $q>1$. For $q=1$, the same
conclusion holds whenever $\psi(\delta)\to0$. In particular,
\[
 \omega_1(I)=|I|,\qquad
 \omega_2(I)=|I|\bigl(\log(e/|I|)\bigr)^{-2}
\]
satisfy
\[
 V_{\omega_1}(\delta)=1,\qquad
 V_{\omega_2}(\delta)\le\bigl(\log(e/\delta)\bigr)^{-2}\longrightarrow0.
\]
A cell bound of the form
\[
 \Box_{i,j}(G)\le C_H(s_i-s_{i-1})(t_{j+1}-t_j)
       \bigl(\log(e/(t_{j+1}-t_j))\bigr)^{-2}
\]
therefore suffices for descent. The logarithmic gain occurs in only one
coordinate. No power improvement in both coordinates is required.

\begin{proposition}[Regular-square criterion]\label{grid:regular}
Let $G_n$ be the grid with both partitions $(k/n)_{k=0}^n$.
Suppose that $\vartheta_H:(0,1]\to[0,\infty)$ satisfies
\[
 \Box_{i,j}(G_n)\le\vartheta_H(1/n),\qquad
 n^2\vartheta_H(1/n)\longrightarrow0.
\]
Then $\Phi(\gamma_0)=\Phi(\gamma_1)$.
\end{proposition}
\begin{proof}
There are $n^2$ cells and $\mesh(G_n)=1/n$. Hence
\[
 0\le\Delta_H(G_n)
 \le n^2\vartheta_H(1/n)\longrightarrow0.
\]
Theorem~\ref{grid:descent} gives the result.
\end{proof}

The regular criterion allows anisotropic exponents that need not make
either interval functional partition-vanishing. For example, a bound
\[
 \Box_{i,j}(G)\le C_H(s_i-s_{i-1})^u(t_{j+1}-t_j)^v,
 \qquad u,v>0,\qquad u+v>2,
\]
gives $\Delta_H(G_n)\le C_Hn^{2-u-v}\to0$. Thus full descent holds
even when $u\le1$ or $v\le1$. The use of a regular sequence fixes the
relative mesh scales and avoids any assertion about arbitrary aspect
ratios.

\begin{theorem}[Superquadratic three-point descent]\label{grid:superquadratic}
If the approximate-action data satisfy $p=1+\eps>2$, sewn transport
is invariant under every relative Lipschitz homotopy. It therefore
defines an action of the full Lipschitz homotopy groupoid on $(M_x)$
by bi-Lipschitz maps. For each representative,
\[
 \Lip(\Phi(\gamma))\le g(\len(\gamma)),\qquad
 \Lip(\Phi(\gamma)^{-1})\le g(\len(\gamma)).
\]
\end{theorem}
\begin{proof}
Consider a cell of $G_n$, with corners
\[
 x=H(s_{i-1},t_j),\quad u=H(s_i,t_j),\quad
 v=H(s_{i-1},t_{j+1}),\quad y=H(s_i,t_{j+1}).
\]
Each of the pairs $(x,u),(u,y),(x,v),(v,y)$ comes from an edge of
Euclidean length $1/n$. All four image distances are at most $L_H/n$.
Comparison of the two routes through their common direct increment
gives
\begin{align*}
 \Box_{i,j}(G_n)
 &\le\dist_{xy}(\mu_{xu}\mu_{uy},\mu_{xy})
       +\dist_{xy}(\mu_{xy},\mu_{xv}\mu_{vy})\\
 &\le\sum_{r=1}^N C_r
  \left(d_P(y,u)^{a_r}d_P(u,x)^{b_r}
       +d_P(y,v)^{a_r}d_P(v,x)^{b_r}\right)\\
 &\le2C L_H^p n^{-p}.
\end{align*}
Consequently,
\begin{equation}\label{grid:superquadratic-rate}
 \Delta_H(G_n)\le2C L_H^p n^{2-p}\longrightarrow0.
\end{equation}
Proposition~\ref{grid:regular} proves equality of the boundary
transports. The concatenation, identity, and reversal formulas in
Proposition~\ref{prop:path-algebra} therefore pass to relative
Lipschitz homotopy classes. The resulting maps are bijections, with
inverse represented by the reversed path. The bounds follow from
\eqref{eq:path-lip} and $\len(\bar\gamma)=\len(\gamma)$.
\end{proof}

The exponent restriction in this theorem enters only in
$n^{2-p}\to0$. At $p=2$ the argument yields a bounded total defect,
which does not imply vanishing. The area example below supplies the
corresponding obstruction and establishes sharpness of the threshold.
The thin theorem itself continues to use the full range $p>1$.

Curry and Manchon's four-point estimate supplies such grid control. We
apply it next, retaining $f,g$ and isolating the linear bound needed
for exponential estimates.

\section{The knitting theorem and its quantitative normalization}
\label{sec:knitting}

The hypothesis of Curry--Manchon's knitting theorem
\cite[Theorem~4.14]{CM} strengthens the defect estimate while retaining
the functions $f$ and $g$ of Definition~\ref{def:approx}. We first give
the conclusion with these original controls. A linear Lipschitz
assumption will be stated separately when an exponential bound is used.

\begin{definition}\label{def:strong-knit}
A strong four-point family satisfies \eqref{eq:diag}, \eqref{eq:f},
and \eqref{eq:chain}. In place of \eqref{eq:defect}, it satisfies
\begin{align}
 \dist_{xy}(\mu_{xu}\mu_{uy},\mu_{xv}\mu_{vy})
 &\le (1+A d_P(u,x))
       \sum_{r=1}^{N}C_r d_P(y,v)^{a_r}d_P(u,v)^{b_r}
       \notag\\
 &\quad+\sum_{r=1}^{N}C_r d_P(x,u)^{b_r}d_P(u,v)^{a_r}
 \label{eq:strong-knit}
\end{align}
for all $x,u,v,y\in P$, where $A\ge0$, $C_r,a_r,b_r>0$, and
$a_r+b_r=2+\kappa$ for a fixed $\kappa>0$. In this section put
\[
 p=2+\kappa,\qquad C=\sum_{r=1}^{N}C_r,\qquad
 K_p=2^p C\zeta(p).
\]
\end{definition}

The specialization $u=z$, $v=y$ in \eqref{eq:strong-knit} gives
\begin{equation}\label{eq:strong-three-knit}
 \dist_{xy}(\mu_{xy},\mu_{xz}\mu_{zy})
 \le\sum_{r=1}^{N}C_r d_P(y,z)^{a_r}d_P(z,x)^{b_r}.
\end{equation}
Indeed, $d_P(y,v)=0$ annihilates the first sum, and
$\mu_{xv}\mu_{vy}=\mu_{xy}\mu_{yy}=\mu_{xy}$.
Thus the family is an approximate action with sewing gain
$p-1=1+\kappa$. All preceding pathwise and thin-groupoid results
apply with that gain. The higher order also gives full homotopy
descent by Theorem~\ref{grid:superquadratic}. The proof below retains
the four-point estimate to display
the direct knitting bound.

\begin{theorem}[Knitting with general controls]\label{thm:knitting}
For a strong four-point family, the formula
\[
 \varphi^{\mathrm{Lip}}_{[\gamma]}=\Phi(\gamma)
\]
defines an action of the relative Lipschitz-homotopy groupoid
$\mathcal G_P^{\mathrm{Lip}}$ by bi-Lipschitz maps.
For a Lipschitz path $\gamma$, write $L_\gamma=\Lip(\gamma)$ and
$\ell_{s,t}=\len(\gamma|_{[s\wedge t,s\vee t]})$. Then
\begin{align}
 \dist_{\gamma(s)\gamma(t)}
  (\varphi^{\mathrm{Lip}}_{[\gamma_{s,t}]},\mu_{\gamma(s)\gamma(t)})
 &\le K_p g(\ell_{s,t})\ell_{s,t}^p
 \le K_p L_\gamma^p g(\ell_{s,t})|t-s|^p,
 \label{eq:knit-local}\\
 \Lip(\varphi^{\mathrm{Lip}}_{[\gamma_{s,t}]})
 &\le g(\ell_{s,t}).\label{eq:knit-lip}
\end{align}
The action is the limit of the original products along every
representative as the partition mesh tends to zero. It is unique
among actions $\Psi$ for which every Lipschitz path $\gamma$ admits
a finite $B_\gamma$ such that
\begin{equation}\label{eq:knit-unique}
 \dist_{\gamma(s)\gamma(t)}
   (\Psi_{[\gamma_{s,t}]},\mu_{\gamma(s)\gamma(t)})
 \le B_\gamma |t-s|^{1+\kappa}
 \qquad(s,t\in[0,1]).
\end{equation}
\end{theorem}

\begin{proof}
Pullback of \eqref{eq:strong-three-knit} along $\gamma$ gives
\[
 \dist(\mu_{\gamma(s)\gamma(t)},
       \mu_{\gamma(s)\gamma(u)}\mu_{\gamma(u)\gamma(t)})
 \le\sum_{r=1}^{N}C_rL_\gamma^p|t-u|^{a_r}|u-s|^{b_r}.
\]
Proposition~\ref{prop:path} therefore gives the partition limits,
\eqref{eq:knit-local}, and \eqref{eq:knit-lip}. The relevant exponent
in the interval theorem is $p$, and its constants are
$C_r^\gamma=C_rL_\gamma^p$. Moreover, \eqref{eq:knit-local} implies
\eqref{eq:knit-unique} for the constructed maps, since
$|t-s|^p\le|t-s|^{1+\kappa}$ and
$g(\ell_{s,t})\le g(\len\gamma)$.

Let $H:[0,1]^2\to P$ be a relative Lipschitz homotopy from
$\gamma_0$ to $\gamma_1$, and let $\ell=\Lip(H)$ for the
Euclidean metric on the square. For $n\ge1$ put
\[
 x_i^j=H(i/n,j/n),\qquad
 R_j=\mu_{x_0^jx_1^j}\cdots\mu_{x_{n-1}^jx_n^j}
 \quad(0\le j\le n).
\]
The relative boundary condition gives $x_0^j=a$, $x_n^j=b$.
For $1\le i\le n$, $0\le j<n$, define
\[
 D_{i,j}=\dist_{x_{i-1}^jx_i^{j+1}}
 \bigl(\mu_{x_{i-1}^jx_i^j}\mu_{x_i^jx_i^{j+1}},
       \mu_{x_{i-1}^jx_{i-1}^{j+1}}
         \mu_{x_{i-1}^{j+1}x_i^{j+1}}\bigr).
\]
Every horizontal row has length at most $\ell$. The strip comparison
of Lemma~\ref{grid:strip}, with its horizontal-prefix bound
$g(\ell)$, gives
\begin{equation}\label{eq:knit-strip}
 \dist_{ab}(R_j,R_{j+1})
 \le g(\ell)\sum_{i=1}^{n}D_{i,j},\qquad
 \dist_{ab}(R_0,R_n)
 \le g(\ell)\sum_{j=0}^{n-1}\sum_{i=1}^{n}D_{i,j}.
\end{equation}

Apply \eqref{eq:strong-knit} with
\[
 x=x_{i-1}^j,\quad u=x_i^j,\quad
 v=x_{i-1}^{j+1},\quad y=x_i^{j+1}.
\]
The Lipschitz estimates are
\[
 d_P(x,u)\le\frac{\ell}{n},\qquad
 d_P(y,v)\le\frac{\ell}{n},\qquad
 d_P(u,v)\le\frac{2\ell}{n}.
\]
Consequently, with
\[
 C_H=\ell^p\left((1+A\ell)
                 \sum_{r=1}^{N}C_r2^{b_r}
                 +\sum_{r=1}^{N}C_r2^{a_r}\right),
\]
we have $D_{i,j}\le C_Hn^{-p}$. There are $n^2$ cells, so
\begin{equation}\label{eq:knit-grid-rate}
 \dist_{ab}(R_0,R_n)
 \le g(\ell)C_Hn^{2-p}
 =g(\ell)C_Hn^{-\kappa}\longrightarrow0.
\end{equation}
The two boundary sewing estimates and the triangle inequality give
\[
 \dist_{ab}(\Phi(\gamma_0),\Phi(\gamma_1))
 \le 2K_p g(\ell)\ell^p n^{1-p}
       +g(\ell)C_Hn^{2-p}.
\]
Both terms tend to zero. All displayed comparison distances are
finite, and separation gives $\Phi(\gamma_0)=\Phi(\gamma_1)$.

Hence $\varphi^{\mathrm{Lip}}$ is well defined. Path concatenation,
constant paths, and reversal give its product, unit, and inverse
identities by Proposition~\ref{prop:path-algebra}. The bounds
\eqref{eq:knit-lip} for a path and its reversal show that every
action map is a Lipschitz homeomorphism.

For uniqueness, let $\Psi$ satisfy \eqref{eq:knit-unique}. Fix
$\gamma$ and a partition $D=(0=t_0<\cdots<t_m=1)$. Its affine
subpaths concatenate to a reparameterization of $\gamma$, so
\[
 \Psi_{[\gamma]}=
 \prod_{j=1}^{m}\Psi_{[\gamma_{t_{j-1},t_j}]}.
\]
The product comparison \eqref{eq:hybrid}, with the $\mu$ factors
as the left prefixes, yields
\begin{align*}
 \dist(\Psi_{[\gamma]},P_\gamma(D))
 &\le g(\len\gamma)B_\gamma
              \sum_{j=1}^{m}(t_j-t_{j-1})^{1+\kappa}\\
 &\le g(\len\gamma)B_\gamma\mesh(D)^\kappa.
\end{align*}
The right-hand side tends to zero. Since
$P_\gamma(D)\to\Phi(\gamma)$, separation of the extended metric
gives $\Psi_{[\gamma]}=\Phi(\gamma)$ for every representative.
No Lipschitz bound on the competing maps is used.
\end{proof}

\begin{corollary}[Linear control]\label{cor:knit-linear}
In addition to Definition~\ref{def:strong-knit}, suppose that
\begin{equation}\label{eq:knit-linear}
 \Lip(\mu_{xy})\le1+B d_P(x,y)\qquad(x,y\in P)
\end{equation}
for some $B\ge0$. The conclusions of Theorem~\ref{thm:knitting}
hold with $g(r)=e^{Br}$. In particular,
\begin{align}
 \Lip(\varphi^{\mathrm{Lip}}_{[\gamma_{s,t}]})
 &\le e^{B\ell_{s,t}}
 \le e^{BL_\gamma|t-s|},\label{eq:knit-exponential}\\
 \dist(\varphi^{\mathrm{Lip}}_{[\gamma_{s,t}]},
               \mu_{\gamma(s)\gamma(t)})
 &\le 2^{2+\kappa}C\zeta(2+\kappa)
       L_\gamma^{2+\kappa}e^{B\ell_{s,t}}|t-s|^{2+\kappa}.
 \label{eq:knit-linear-local}
\end{align}
\end{corollary}
\begin{proof}
For every finite chain,
\[
 \Lip(\mu_{x_0x_1}\cdots\mu_{x_{m-1}x_m})
 \le\prod_{j=1}^{m}(1+B d_P(x_{j-1},x_j))
 \le\exp\left(B\sum_{j=1}^{m}d_P(x_{j-1},x_j)\right).
\]
Thus the chain control can be replaced by $e^{Br}$ in every preceding
estimate. Substitution of $p=2+\kappa$ proves both assertions.
\end{proof}

The constants $A$ and $B$ need not coincide. Replacing each by
$\max\{A,B\}$ gives a common constant without changing the class
having both bounds. The exponential conclusion uses
\eqref{eq:knit-linear}, whereas qualitative homotopy descent uses
only the original strong hypotheses. The factor
$L_\gamma^{2+\kappa}$ in \eqref{eq:knit-linear-local} records
pullback scaling. The power $2^{2+\kappa}$ and
$\zeta(2+\kappa)$ are the constants supplied by the interval
deletion estimate at total order $2+\kappa$. They provide a valid
normalization of the local bound in \cite[equation~(27)]{CM}.

\begin{proposition}\label{prop:knit-nonlinear}
The hypotheses of Definition~\ref{def:strong-knit}, with general
$f$ and $g$, do not imply the exponential bound
\eqref{eq:knit-exponential} for any finite constant $B$.
\end{proposition}
\begin{proof}
Take $P=[0,1]$, $M_x=\R$, and
\[
 \mu_{xy}(v)=e^{\sqrt{x}-\sqrt{y}}v.
\]
Every chain product equals $\mu_{x_0x_m}$, so every three-point
and four-point defect vanishes. Since
$|\sqrt{x}-\sqrt{y}|\le\sqrt{|x-y|}$, the original controls hold
with $f(r)=e^{\sqrt r}-1$ and $g(r)=e^{\sqrt r}$.
The four-point inequality holds with any $A\ge0$ and any positive
constants of the prescribed orders. For $\gamma(t)=t$ the sewn
transport is the exact action, and
\[
 \Lip(\Phi(\gamma_{t,0}))=e^{\sqrt t}.
\]
For every finite $B\ge0$, this exceeds $e^{Bt}$ when
$0<t<\min\{1,(B+1)^{-2}\}$. Thus the exponential assertion in
\cite[equation~(28)]{CM} requires an additional hypothesis, while
Theorem~\ref{thm:knitting} retains the full qualitative statement
with general $f$ and $g$.
\end{proof}

\begin{corollary}[Flat holonomy under the local hypotheses]
\label{cor:remark415}
For every approximate action, the maps in
Theorem~\ref{holo:representation} give a consistent based thin
holonomy. Suppose, in addition, that at least one of the following
conditions holds:
\begin{enumerate}[label=\textup{(\alph*)}]
\item the three-point order is $p>2$,
\item the family satisfies Definition~\ref{def:strong-knit},
\item for each relative Lipschitz contraction $H$ of a based loop,
there is a mesh-null family of grids with $\Delta_H(G)\to0$.
\end{enumerate}
Then, for every $a\in P$,
\begin{equation}\label{eq:flat-holonomy}
 \operatorname{Hol}_a(\ker q_a)=\{\Id_{M_a}\},\qquad
 \operatorname{Hol}_a
       =\widehat{\operatorname{Hol}}_a\circ q_a,
\end{equation}
where the factored homomorphism on $\pi_1^{\mathrm{Lip}}(P,a)$
is unique. The factored action is independent of every relative
Lipschitz-homotopy representative and has the same partition formula.
Under the strong hypotheses and the linear bound
\eqref{eq:knit-linear}, a full based class $\widehat h$ satisfies
\begin{equation}\label{eq:flat-holonomy-bound}
 \begin{split}
 \max\{\Lip(\widehat{\operatorname{Hol}}_a(\widehat h)),
          \Lip(\widehat{\operatorname{Hol}}_a(\widehat h)^{-1})\}
 &\le e^{B\widehat L(\widehat h)},\\
 \operatorname{disp}_a(\widehat{\operatorname{Hol}}_a(\widehat h))
 &\le K_p e^{B\widehat L(\widehat h)}
                 \widehat L(\widehat h)^p.
 \end{split}
\end{equation}
Here $p=2+\kappa$, $K_p=2^pC\zeta(p)$, and
$\widehat L(\widehat h)=\inf_{\lambda\in\widehat h}\len(\lambda)$.
\end{corollary}
\begin{proof}
The thin representation is already constructed under the basic
assumptions. Let $n\in\ker q_a$, let $\lambda$ represent it, and
let $H$ be a relative Lipschitz contraction of $\lambda$ to $c_a$.
Under (a), Theorem~\ref{grid:superquadratic} gives
$\Phi(\lambda)=\Phi(c_a)$. Under (b), Theorem~\ref{thm:knitting}
gives the same equality. Under (c), Theorem~\ref{grid:descent}
applied to the stipulated family gives that equality. In each case,
\[
 \operatorname{Hol}_a(n)=\Phi(\lambda)=\Phi(c_a)=\Id_{M_a}.
\]
Theorem~\ref{holo:flatness} therefore gives the unique factorizations
and \eqref{eq:flat-holonomy}. In the linear case,
Corollary~\ref{cor:knit-linear} gives the first bound for each loop
representative and its reverse. Its intrinsic local bound at
$(s,t)=(0,1)$ gives the displacement bound, using $\mu_{aa}=\Id$.
The maps are independent of the relative Lipschitz-homotopy representative, and the
functions $r\mapsto e^{Br}$ and $r\mapsto K_pe^{Br}r^p$ are
continuous and nondecreasing. Infima over representatives give
\eqref{eq:flat-holonomy-bound}.
\end{proof}

This establishes both assertions of \cite[Remark~4.15]{CM} in the
relative Lipschitz category: the conjectural sewing construction
gives based holonomy, and knitting makes that holonomy trivial on
the relatively Lipschitz-contractible loops. Conditions (a) and
(b) verify flatness directly from local increment estimates.
Condition (c) is the separate sufficient comparison criterion,
whose concrete scalar hypotheses were given in
Section~\ref{grid:section}.

In the next section, two integral models will test and reinforce the conclusions present until now and in the introduction: an area example at order
two, and a disk metric separating Lipschitz contractions from continuous
ones even under strong knitting.

\section{The sharp obstruction to flat holonomy}
\label{sec:curvature}

The first model uses chord integrals of $x_1\,dx_2$. Its sewn transport
records signed area around a loop. Computing both its three-point defect
and its holonomy will locate the threshold in
Theorem~\ref{grid:superquadratic}.

For $c\in\R$, write $T_c(u)=u+c$. Then
\begin{equation}\label{curvature:translations}
 T_cT_d=T_{c+d},\qquad T_c^{-1}=T_{-c},\qquad
 \Lip(T_c)=1,\qquad \dist(T_c,T_d)=|c-d|.
\end{equation}
All fibers in this section are the ordinary complete metric space $\R$.

\begin{proposition}\label{curvature:model}
On Euclidean $P=\R^2$, define
\begin{equation}\label{curvature:beta}
 \beta(x,y)=\frac{x_1+y_1}{2}(y_2-x_2),\qquad
 \mu_{xy}=T_{\beta(x,y)}.
\end{equation}
These maps form an approximate action with
\[
 f=0,\qquad g=1,\qquad N=1,\qquad
 C_1=\frac12,\qquad a_1=b_1=1,\qquad \eps=1.
\]
They have exact pairwise inverses. For every Lipschitz
$\gamma=(\gamma_1,\gamma_2):[0,1]\to\R^2$ and all $s,t\in[0,1]$,
\begin{equation}\label{curvature:integral}
 \varphi_{[\gamma_{s,t}]}=T_{I_{s,t}(\gamma)},\qquad
 I_{s,t}(\gamma)=\int_s^t\gamma_1(r)\gamma_2'(r)\,dr,
\end{equation}
where the integral is a signed Lebesgue integral when $s>t$.
\end{proposition}
\begin{proof}
The diagonal and reversal identities are
\[
 \beta(x,x)=0,\qquad \beta(y,x)=-\beta(x,y).
\]
Every finite product of increments is a translation, proving the
Lipschitz assumptions with $f=0$ and $g=1$. Direct expansion gives
\begin{equation}\label{curvature:triangle}
 \beta(x,y)-\beta(x,z)-\beta(z,y)
 =-\frac12\det(z-x,y-z).
\end{equation}
Consequently
\[
 \dist_{xy}(\mu_{xy},\mu_{xz}\mu_{zy})
 =\frac12|\det(z-x,y-z)|
 \le\frac12|z-x|\,|y-z|.
\]
This verifies every approximate-action assumption.

Put $L=\Lip(\gamma)$. The coordinate functions of $\gamma$ are
absolutely continuous, have derivatives bounded in absolute value by
$L$ almost everywhere, and satisfy the fundamental theorem of calculus
on every subinterval. For a partition
$D=(0=t_0<\cdots<t_n=1)$, write $h_j=t_j-t_{j-1}$. Its product is
translation by
\[
 S_D=\sum_{j=1}^n
 \frac{\gamma_1(t_{j-1})+\gamma_1(t_j)}2
       \bigl(\gamma_2(t_j)-\gamma_2(t_{j-1})\bigr).
\]
For $r\in[t_{j-1},t_j]$,
\[
 \left|\frac{\gamma_1(t_{j-1})+\gamma_1(t_j)}2-\gamma_1(r)\right|
 \le\frac L2\bigl(r-t_{j-1}+t_j-r\bigr)=\frac L2h_j.
\]
The fundamental theorem and $|\gamma_2'|\le L$ yield
\begin{align*}
 \left|S_D-\int_0^1\gamma_1(r)\gamma_2'(r)\,dr\right|
 &\le\sum_{j=1}^n\int_{t_{j-1}}^{t_j}
       \frac L2h_j|\gamma_2'(r)|\,dr\\
 &\le\frac{L^2}{2}\sum_{j=1}^nh_j^2
 \le\frac{L^2}{2}\mesh(D).
\end{align*}
Equation~\eqref{curvature:translations} identifies this scalar limit
with the sewn limit. For $s\ne t$, the affine subpath has derivative
\[
 (\gamma_{s,t})'(u)=(t-s)\gamma'(s+(t-s)u)
 \quad\text{for almost every }u.
\]
Applying the proved formula to this path and making the signed affine
change of variable gives \eqref{curvature:integral}. For $s=t$, the
path is constant and both sides are identities.
\end{proof}

\begin{theorem}\label{curvature:nonflat}
The approximate action \eqref{curvature:beta} has nontrivial sewn
holonomy on a relatively Lipschitz-contractible loop. This holds both
on $\R^2$ and on its Euclidean unit square. Its thin holonomy does not
factor through relative Lipschitz homotopy classes.
\end{theorem}
\begin{proof}
Let $\lambda$ traverse the unit-square boundary counterclockwise from
the origin, with constant speed on each of its four sides:
\begin{equation}\label{curvature:square-loop}
 \lambda(r)=
 \begin{cases}
 (4r,0),&0\le r\le\tfrac14,\\
 (1,4r-1),&\tfrac14\le r\le\tfrac12,\\
 (3-4r,1),&\tfrac12\le r\le\tfrac34,\\
 (0,4-4r),&\tfrac34\le r\le1.
 \end{cases}
\end{equation}
The four affine formulas agree at their joining points and have speed
$4$, so $\Lip(\lambda)=4$. On the first and third sides
$\lambda_2'=0$. On the fourth side $\lambda_1=0$. Hence
\begin{equation}\label{curvature:holonomy-one}
 I_{0,1}(\lambda)=\int_{1/4}^{1/2}4\,dr=1,\qquad
 \varphi_{[\lambda]}=T_1\ne\Id.
\end{equation}
Define
\[
 H(r,t)=(1-t)\lambda(r).
\]
Convexity places its image in $[0,1]^2$. Its boundary values satisfy
\[
 H(r,0)=\lambda(r),\quad H(r,1)=0,\quad
 H(0,t)=H(1,t)=0.
\]
For two domain points,
\begin{align*}
 |H(r,t)-H(r',t')|
 &\le4|r-r'|+\sqrt2|t-t'|\\
 &\le\sqrt{18}\sqrt{|r-r'|^2+|t-t'|^2}.
\end{align*}
Thus $H$ is a relative Lipschitz contraction. Its constant boundary
loop has identity transport, whereas \eqref{curvature:holonomy-one}
is nonidentity. Therefore the action cannot factor through the full
relative Lipschitz homotopy quotient.

In fact, every Lipschitz loop $\eta$ based at the origin in either
of these spaces admits the contraction $(1-t)\eta(r)$. If
$L_\eta=\Lip(\eta)$ and $R_\eta=\sup_r|\eta(r)|$, its Lipschitz
constant is at most $\sqrt{L_\eta^2+R_\eta^2}$. Their full based
Lipschitz loop groups are therefore trivial, although the corresponding
thin holonomy homomorphisms contain $T_1$ in their images.
\end{proof}

\begin{proposition}\label{curvature:area}
The rectangular comparison defect of \eqref{curvature:beta} on an
axis-parallel rectangle is exactly its Euclidean area. For the radial
contraction in Theorem~\ref{curvature:nonflat}, no family of grids
whose horizontal mesh tends to zero can have total rectangular
comparison defect tending to zero.
\end{proposition}
\begin{proof}
For corners
\[
 x=(a,c),\quad u=(b,c),\quad v=(a,d),\quad y=(b,d),
\]
one has
\[
 \mu_{xu}\mu_{uy}=T_{b(d-c)},\qquad
 \mu_{xv}\mu_{vy}=T_{a(d-c)}.
\]
Consequently
\begin{equation}\label{curvature:rectangle}
 \dist_{xy}(\mu_{xu}\mu_{uy},\mu_{xv}\mu_{vy})
 =|b-a|\,|d-c|.
\end{equation}
The $n^2$ cells of a regular spatial grid on the unit square thus
have total defect $n^2n^{-2}=1$.

For the actual relative contraction $H$, take partitions
$0=r_0<\cdots<r_n=1$ and $0=t_0<\cdots<t_m=1$ and set
$x_{ij}=H(r_i,t_j)$. Define the signed cell defects by
\begin{align*}
 E_{ij}={}&\beta(x_{i-1,j-1},x_{i,j-1})
          +\beta(x_{i,j-1},x_{ij})\\
         &-\beta(x_{i-1,j-1},x_{i-1,j})
          -\beta(x_{i-1,j},x_{ij}).
\end{align*}
The unsigned rectangular comparison defect is $|E_{ij}|$.
Summing first in $i$ and then in $j$ cancels all interior edges.
Both vertical boundary paths and the top path are constant, giving
\[
 \sum_{j=1}^m\sum_{i=1}^nE_{ij}
 =\sum_{i=1}^n\beta(\lambda(r_{i-1}),\lambda(r_i))=S_D.
\]
It follows that
\begin{equation}\label{curvature:grid-certificate}
 \sum_{j=1}^m\sum_{i=1}^n|E_{ij}|\ge|S_D|,
 \qquad
 \liminf_{\mesh(D)\to0}\sum_{j=1}^m\sum_{i=1}^n|E_{ij}|\ge1.
\end{equation}
Here the last inequality follows from
Proposition~\ref{curvature:model} and
\eqref{curvature:holonomy-one}, independently of the vertical
partition. It applies in particular to every family whose two meshes
tend to zero.
\end{proof}

\begin{theorem}\label{curvature:sharp}
For every prescribed $\eps_0\in(0,1]$, the restriction of
\eqref{curvature:beta} to $P=[0,1]^2$ satisfies the three-point
hypothesis with gain $\eps_0$ and has nonflat sewn holonomy.
It admits no three-point presentation with gain strictly greater
than one.
\end{theorem}
\begin{proof}
Put $D_0=\sqrt2$ and $a=b=(1+\eps_0)/2$. These exponents are
positive and at most one. For $0<r,s\le D_0$,
\[
 rs=r^as^b r^{1-a}s^{1-b}
 \le D_0^{1-\eps_0}r^as^b.
\]
The same inequality holds if either distance vanishes. The defect
bound of Proposition~\ref{curvature:model} therefore supplies
\[
 N=1,\qquad C_1=\frac12D_0^{1-\eps_0},\qquad
 a_1=b_1=\frac{1+\eps_0}{2}.
\]
The loop and contraction from Theorem~\ref{curvature:nonflat}
remain in this square and still have holonomy $T_1$.

Suppose a finite presentation existed with $\eps'>1$ and
$a_r+b_r=1+\eps'$ for every $r$. For $0<h\le1$, set
\[
 x=(0,0),\qquad z=(h,0),\qquad y=(h,h).
\]
Equation~\eqref{curvature:triangle} gives defect $h^2/2$, and both
successive distances are $h$. The proposed estimate would imply
\[
 \frac12h^2\le\sum_r C_r h^{1+\eps'},\qquad
 \frac12\le\left(\sum_r C_r\right)h^{\eps'-1}
 \longrightarrow0.
\]
This contradiction excludes every such finite presentation.
\end{proof}

Combined with the superquadratic descent theorem, this proves that
total three-point order strictly greater than two is the sharp
universal sufficient threshold for invariance under relative
Lipschitz homotopy. At every smaller admissible order, thin holonomy
still exists, but the example has nontrivial holonomy on a relatively
Lipschitz-contractible loop. This distinction concerns Lipschitz
homotopies throughout. An identification with the ordinary
topological fundamental group requires a separate comparison
between continuous and Lipschitz homotopies of the metric space.

\subsection{Continuous contractibility and the metric hypothesis}

The distinction between continuous and Lipschitz homotopies persists
even under the strong knitting assumptions. It cannot be removed by
retaining only the topology of the parameter space.

\begin{proposition}\label{curvature:continuous-scope}
There is a compact metric space homeomorphic to the closed disk and
an approximate action with three-point gain $\eps=2$, satisfying
the strong four-point estimate with linear control constant zero,
whose sewn holonomy is nontrivial on a Lipschitz loop admitting a
continuous relative contraction.
\end{proposition}
\begin{proof}
Let $P=\{x\in\R^2:|x|\le1\}$ with metric
\[
 d(x,y)=|x-y|+\bigl||x|-|y|\bigr|^{1/2}.
\]
The second summand is a pseudometric, and the first separates points,
so $d$ is a metric. Euclidean convergence implies $d$-convergence and
the reverse implication follows from $|x-y|\le d(x,y)$. Thus $P$
has the usual topology of the closed disk.

Define $\chi(s)=((4s-1)_+/3)^3$ and the $C^2$ one-form
\[
 A=\chi(|x|^2)(x_1\,dx_2-x_2\,dx_1).
\]
Its exterior derivative on the disk is
\[
 dA=B(x)\,dx_1\wedge dx_2,\qquad
 B(x)=2\chi(|x|^2)+2|x|^2\chi'(|x|^2).
\]
Here $B=0$ for $|x|\le1/2$ and $|B|\le10$ on $P$, since
$0\le\chi\le1$ and $0\le\chi'\le4$ on $[0,1]$. Put
\[
 \mu_{xy}=T_{b(x,y)},\qquad
 b(x,y)=\int_0^1 A(x+u(y-x))[y-x]\,du.
\]
The diagonal, inverse, and chain-Lipschitz assumptions hold with
$f=0$ and $g=1$. The elementary Stokes formula on an affine triangle
gives
\begin{equation}\label{curvature:continuous-triangle}
 |b(x,y)-b(x,z)-b(z,y)|
 \le5|\det(x-z,y-z)|.
\end{equation}
Write $v=x-z$, $w=y-z$, $d_1=d(x,z)$, and $d_2=d(z,y)$.
If $r=|z|\ge1/4$, set $n=z/r$ and let $n^\perp$ be a perpendicular
unit vector. Then
\[
 n\cdot v=\frac{|x|^2-|z|^2-|v|^2}{2r},\qquad
 \bigl||x|^2-|z|^2\bigr|\le2d_1^2,
\]
so $|n\cdot v|\le6d_1^2$ and $|n^\perp\cdot v|\le d_1$.
The corresponding bounds hold for $w$, giving
\[
 |\det(v,w)|\le6(d_1^2d_2+d_1d_2^2).
\]
If $|z|<1/4$ and $|x|,|y|\le1/2$, the triangle lies in the region
where $A=0$, and the defect is zero. Otherwise $d_1+d_2\ge1/4$,
and
\[
 |\det(v,w)|\le d_1d_2
 \le4(d_1^2d_2+d_1d_2^2).
\]
Together with \eqref{curvature:continuous-triangle}, these estimates
prove the global three-point bound
\[
 \dist(\mu_{xy},\mu_{xz}\mu_{zy})
 \le30d(x,z)^2d(z,y)+30d(x,z)d(z,y)^2.
\]
It has two positive-exponent terms of total order $3$, hence gain $2$.
For the four-point condition, introduce
\[
 R(x,z,y)=b(x,y)-b(x,z)-b(z,y).
\]
The algebraic identity
\[
 b(x,u)+b(u,y)-b(x,v)-b(v,y)
 =-R(x,u,v)+R(u,v,y)
\]
gives
\begin{align*}
 \dist(\mu_{xu}\mu_{uy},\mu_{xv}\mu_{vy})
 \le30\bigl(&d(y,v)d(u,v)^2+d(y,v)^2d(u,v)\\
            &+d(x,u)^2d(u,v)+d(x,u)d(u,v)^2\bigr).
\end{align*}
This is the strong estimate with exponent pairs $(1,2)$ and $(2,1)$,
total order $3$, and prefactor $1$. Also $\Lip(\mu_{xy})=1$, so
its linear Lipschitz-control constant is zero.

The loop $\lambda(t)=(\cos(2\pi t),\sin(2\pi t))$ is
$2\pi$-Lipschitz for $d$, since its radius is constant. We verify its
sewn value directly. For two successive angles with positive
difference $\theta$, the chord formula gives
\[
 b(\lambda(s),\lambda(t))
 =\sin\theta\int_0^1\chi\bigl(1-2u(1-u)(1-\cos\theta)\bigr)\,du.
\]
The arguments of $\chi$ belong to $[0,1]$, and its Lipschitz
constant there is at most $4$. Since $\chi(1)=1$,
\[
 |b(\lambda(s),\lambda(t))-\theta|
 \le|\sin\theta-\theta|+2|\sin\theta|(1-\cos\theta)
 \le\frac76\theta^3.
\]
For a partition with angular increments $\theta_j$, summation yields
\[
 \left|\sum_jb(\lambda(t_{j-1}),\lambda(t_j))-2\pi\right|
 \le\frac76\sum_j\theta_j^3
 \le\frac{7\pi}{3}\bigl(2\pi\mesh(D)\bigr)^2\longrightarrow0.
\]
Thus $\varphi_{[\lambda]}=T_{2\pi}\ne\Id$.

For $a=(1,0)$ the affine contraction
\[
 H(s,t)=(1-t)\lambda(s)+ta
\]
stays in $P$, is continuous for $d$, fixes both endpoints at $a$,
and joins $\lambda$ to $c_a$. It is therefore a continuous relative
contraction with nontrivial sewn boundary holonomy.

There is no conflict with Lipschitz-homotopy descent. Every
$d$-Lipschitz path $\eta$ has constant radius. Indeed, a partition of
$[s,t]$ into $n$ equal intervals gives
\[
 \bigl||\eta(t)|-|\eta(s)|\bigr|
 \le n\Lip(\eta)^2\left(\frac{t-s}{n}\right)^2
 =\frac{\Lip(\eta)^2(t-s)^2}{n}\longrightarrow0.
\]
In particular, the displayed contraction is not Lipschitz for $d$.
\end{proof}

Translation increments also connect transport with integration. We now
replace the preceding chord increments by controlled rough increments
and apply thin descent and exact tree transport to their compensated sums.

\section{Controlled integration on the metric thin groupoid}
\label{sec:integration}

We use Gubinelli's compensated sums \cite{Gubinelli,FH} for controlled
fields on a metric space. Theorem~\ref{thm:main} makes their limits
agree on thin-equivalent paths. Theorem~\ref{thm:tree} gives primitives
on trees and then allows paths of possibly infinite length.
Substitution keeps a common enhanced controller throughout, following
the interval convention in \cite[Remark~4.12]{FH}. The related
transitivity result for cocyclic one-forms is
\cite[Proposition~40]{LY}.

Let $V,E,F$ be Banach spaces, with completed projective tensor products.
Fix $1/3<\alpha\le1$, $q=3\alpha$, and
$c_\alpha=2^{3\alpha}\zeta(3\alpha)$. For a zero-diagonal two-point
map $B$, put $[B]_\beta=\sup_{x\ne y}\|B_{xy}\|/d(x,y)^\beta$.
For a one-point map this notation applies to its increments.
An enhanced field $\mathbf X=(X,\mathbb X)$ on $P$ consists of
zero-diagonal maps with values in $V$ and $V\widehat\otimes_\pi V$,
respectively, satisfying, for every $x,y,z\in P$,
\begin{equation}\label{int:chen}
 X_{xz}=X_{xy}+X_{yz},\qquad
 \mathbb X_{xz}=\mathbb X_{xy}+\mathbb X_{yz}+X_{xy}\otimes X_{yz},
 \qquad M_1=[X]_\alpha<\infty,\quad M_2=[\mathbb X]_{2\alpha}<\infty.
\end{equation}
The second level is part of the data. A controlled pair is
$Y:P\to\mathcal L(V,E)$,
$Y':P\to\mathcal L(V,\mathcal L(V,E))$ such that
\begin{equation}\label{int:controlled}
 R^Y_{xy}=Y_y-Y_x-Y'_x(X_{xy},\cdot),\qquad
 [R^Y]_{2\alpha}+[Y']_\alpha<\infty.
\end{equation}
Here $Y'_x(u,v)=(Y'_xu)v$, extended to the projective tensor product.
The derivative $Y'$ is specified even when it is not uniquely determined
by $Y$. All notation involving $Y$ below retains this derivative. Set
\begin{equation}\label{int:increment}
 A^Y_{xy}=Y_xX_{xy}+Y'_x\mathbb X_{xy},\qquad
 D_Y=M_1[R^Y]_{2\alpha}+M_2[Y']_\alpha.
\end{equation}

\begin{theorem}[Metric controlled integral]\label{int:main}
There is an additive functor $\mathcal I_{\mathbf X}(Y):\Gthin\to(E,+)$
given, for $x_i=\gamma(t_i)$, by
\begin{equation}\label{int:limit}
 \mathcal I_{\mathbf X}(Y)([\gamma])
 =\lim_{\mesh(D)\to0}\sum_i A^Y_{x_i x_{i+1}},\qquad
 \left\|\mathcal I_{\mathbf X}(Y)([\gamma])-\sum_i A^Y_{x_i x_{i+1}}\right\|
 \le c_\alpha D_Y\len(\gamma)\omega_\gamma(D)^{q-1}.
\end{equation}
For a class $h$ represented by a path from $a$ to $b$,
\begin{equation}\label{int:class}
 \|\mathcal I_{\mathbf X}(Y)(h)-A^Y_{ab}\|
 \le c_\alpha D_Y\mathcal L(h)^q.
\end{equation}
It is unique among additive functors $F$ such that, on each Lipschitz
path $\gamma$, some $C_\gamma<\infty$ and $\theta_\gamma>1$ satisfy
$\|F([\gamma_{s,t}])-A^Y_{\gamma(s)\gamma(t)}\|
\le C_\gamma|t-s|^{\theta_\gamma}$ for all $s,t$.
It is linear in the
controlled pair and commutes with bounded linear maps on $E$ and
Lipschitz pullbacks of the parameter space.
\end{theorem}
\begin{proof}
Expansion using \eqref{int:chen}--\eqref{int:controlled} gives
\begin{align}\label{int:defect}
 A^Y_{xz}-A^Y_{xy}-A^Y_{yz}
 &=-R^Y_{xy}X_{yz}-(Y'_y-Y'_x)\mathbb X_{yz},\\
 \|A^Y_{xz}-A^Y_{xy}-A^Y_{yz}\|
 &\le M_1[R^Y]_{2\alpha}d(x,y)^{2\alpha}d(y,z)^\alpha
     +M_2[Y']_\alpha d(x,y)^\alpha d(y,z)^{2\alpha}.\notag
\end{align}
On the constant fiber $E$, let $\mu_{xy}(v)=v+A^Y_{xy}$. These
translations satisfy Definition~\ref{def:approx} with $f=0$, $g=1$,
and total defect order $q>1$. Zero defect terms may be omitted. If
$D_Y=0$, \eqref{int:defect} gives exact additivity and every sum
telescopes. Otherwise Theorem~\ref{thm:main} applies, and its limit
is the translation by the limit of the sums. Lemma~\ref{lem:mesh-limit}
gives \eqref{int:limit}. The local estimate, followed by the infimum
over representatives, gives \eqref{int:class}. The uniqueness assertion
is that of Theorem~\ref{thm:main}, applied to translations. Each
partition sum is linear in $(Y,Y')$. If $f:Q\to P$ is Lipschitz,
pullback multiplies each order-$\beta$ seminorm by at most
$\Lip(f)^\beta$ and maps thin contractions to thin contractions.
The corresponding sums agree term by term. The same argument applies
to a bounded linear map on $E$, which commutes with their limits.
\end{proof}

For any path $\eta$, write $X^\eta_{st}=X_{\eta(s)\eta(t)}$,
$\mathbb X^\eta_{st}=\mathbb X_{\eta(s)\eta(t)}$, and
$x^\eta_t=X_{\eta(0)\eta(t)}$. Then
$x^\eta_t-x^\eta_s=X^\eta_{st}$. On a compact set $S\subset P$,
$Y'$ is bounded and, with $r=\diam S$,
\begin{equation}\label{int:compact-control}
 [Y]_{\alpha;S}\le\|Y'\|_{\infty;S}M_1+[R^Y]_{2\alpha}r^\alpha.
\end{equation}
Consequently Lipschitz pullback gives interval controlled data.
For $\alpha\le1/2$, \eqref{int:limit} agrees with Gubinelli
integration \cite[Theorem~4.10]{FH}.
For $\alpha>1/2$ the second-level sum tends
to zero, since it is bounded by
$\sup_t\|Y'_{\gamma(t)}\|M_2\Lip(\gamma)^{2\alpha}
\mesh(D)^{2\alpha-1}$, and the integral is Young's.
An interval enhancement extends to reversed pairs by
$X_{ts}=-X_{st}$ and
$\mathbb X_{ts}=-\mathbb X_{st}+X_{st}\otimes X_{st}$.
Then $R^Y_{ts}=-R^Y_{st}+(Y'_t-Y'_s)(X_{st},\cdot)$,
The anchored representation $b_t=X_{0t}$, $C_t=\mathbb X_{0t}$,
$X_{st}=b_t-b_s$,
$\mathbb X_{st}=C_t-C_s-b_s\otimes(b_t-b_s)$ verifies Chen for
all triples. The reverse estimates increase $M_2$ by at most $M_1^2$
and $[R^Y]_{2\alpha}$ by at most $[Y']_\alpha M_1$.
Thus the all-pairs convention includes the usual interval setting.
If $3\alpha>2$, Theorem~\ref{grid:superquadratic} additionally gives
full relative Lipschitz-homotopy invariance. The same conclusion
holds whenever the translations satisfy Definition~\ref{def:strong-knit},
by Theorem~\ref{thm:knitting}.

\begin{proposition}[Joint stability]\label{int:stability}
Let $(\widetilde{\mathbf X},\widetilde Y,\widetilde Y')$ be another
enhanced field and controlled pair on the same spaces. With each
remainder computed against its own driver, put
\begin{align*}
 D_\Delta={}&[X]_\alpha[R^Y-R^{\widetilde Y}]_{2\alpha}
 +[R^{\widetilde Y}]_{2\alpha}[X-\widetilde X]_\alpha\\
 &+[\mathbb X]_{2\alpha}[Y'-\widetilde Y']_\alpha
 +[\widetilde Y']_\alpha[\mathbb X-\widetilde{\mathbb X}]_{2\alpha}.
\end{align*}
For a class $h$ represented by a path from $a$ to $b$,
\begin{equation}\label{int:stable-bound}
 \|\mathcal I_{\mathbf X}(Y)(h)
   -\mathcal I_{\widetilde{\mathbf X}}(\widetilde Y)(h)
   -(A^Y_{ab}-A^{\widetilde Y}_{ab})\|
 \le c_\alpha D_\Delta\mathcal L(h)^q.
\end{equation}
\end{proposition}
\begin{proof}
Writing $\delta A_{xyz}=A_{xz}-A_{xy}-A_{yz}$ and
$\Delta Y'_{xy}=Y'_y-Y'_x$, subtraction of \eqref{int:defect} gives
\begin{align*}
 \delta(A^Y-A^{\widetilde Y})_{xyz}
 ={}&-(R^Y-R^{\widetilde Y})_{xy}X_{yz}
      -R^{\widetilde Y}_{xy}(X-\widetilde X)_{yz}\\
    &-\Delta(Y'-\widetilde Y')_{xy}\mathbb X_{yz}
      -\Delta\widetilde Y'_{xy}
                 (\mathbb X-\widetilde{\mathbb X})_{yz}.
\end{align*}
The four coefficients sum to $D_\Delta$, with mixed orders
$(2\alpha,\alpha)$ or $(\alpha,2\alpha)$. Sewing translations by
$A^Y-A^{\widetilde Y}$ gives \eqref{int:stable-bound}, since their
partition sums converge to the difference of the integrals. If
$D_\Delta=0$, exact telescoping gives the same conclusion.
\end{proof}

\begin{proposition}[Substitution with a common controller]\label{int:substitution}
Let $H:P\to\mathcal L(E,F)$ be controlled by $X$, with derivative
$H':P\to\mathcal L(V,\mathcal L(E,F))$ and remainder $R^H$.
Assume $Y,Y',H,H'$ are bounded and $\alpha$-H\"older.
The product $Z=HY$ is controlled, with derivative
\begin{equation}\label{int:product-derivative}
 Z'_x(u,v)=H'_x(u)(Y_xv)+H_xY'_x(u,v).
\end{equation}
For a Lipschitz path $\gamma$, set
$J^\gamma_t=\mathcal I_{\mathbf X}(Y)([\gamma_{0,t}])$ and
$B_x(u,v)=H'_x(u)(Y_xv)$. The common-controller integral satisfies
\begin{equation}\label{int:substitution-identity}
 \int_\gamma H\,d_{\mathbf X}J^\gamma
 \;:=\lim_{\mesh(D)\to0}\sum_i
 \bigl(H_{x_i}(J^\gamma_{t_{i+1}}-J^\gamma_{t_i})
              +B_{x_i}\mathbb X_{x_i x_{i+1}}\bigr)
 =\mathcal I_{\mathbf X}(HY)([\gamma]).
\end{equation}
The value depends only
on the thin class of $\gamma$. No primitive on $P$ is assumed.
\end{proposition}
\begin{proof}
With $\Delta H_{xy}=H_y-H_x$ and $\Delta Y_{xy}=Y_y-Y_x$,
\begin{align*}
 R^Z_{xy}&=H_xR^Y_{xy}+R^H_{xy}Y_x+\Delta H_{xy}\Delta Y_{xy},\\
 [R^Z]_{2\alpha}
 &\le\|H\|_\infty[R^Y]_{2\alpha}
       +[R^H]_{2\alpha}\|Y\|_\infty+[H]_\alpha[Y]_\alpha,\\
 [Z']_\alpha
 &\le[H']_\alpha\|Y\|_\infty+\|H'\|_\infty[Y]_\alpha
       +[H]_\alpha\|Y'\|_\infty+\|H\|_\infty[Y']_\alpha.
\end{align*}
Additivity and \eqref{int:class} imply
\begin{equation}\label{int:J-error}
 J^\gamma_t-J^\gamma_s=A^Y_{\gamma(s)\gamma(t)}+E_{s,t},
 \qquad\|E_{s,t}\|\le c_\alpha D_Y\ell_{s,t}^{q}.
\end{equation}
Thus $J^\gamma$ is controlled by $x^\gamma$ with derivative
$Y\circ\gamma$: its remainder has order $2\alpha$, since
$q\ge2\alpha$ and $\ell_{s,t}\le\Lip(\gamma)|t-s|$.
For every $D$, the difference between the sum in
\eqref{int:substitution-identity} and $\sum_iA^Z_{x_i x_{i+1}}$ is
$\sum_iH_{x_i}E_{t_i,t_{i+1}}$. Its norm is at most
\[
 \|H\|_\infty c_\alpha D_Y\len(\gamma)
                  \omega_\gamma(D)^{q-1}\longrightarrow0.
\]
Theorem~\ref{int:main} applied to $(Z,Z')$ proves the identity and
thin invariance.
\end{proof}

The subscript in $d_{\mathbf X}J$ records the retained enhancement
and derivatives. For another bounded controlled field
$K:P\to\mathcal L(F,G)$, with $G$ Banach, satisfying the same
regularity assumptions, both bracketings of $KHY$ have derivative
\[
 (KHY)'_x(u,v)=K'_x(u)H_xY_xv
             +K_xH'_x(u)Y_xv+K_xH_xY'_x(u,v).
\]
The inner-integral path is
$Q^\gamma_t=\mathcal I_{\mathbf X}(HY)([\gamma_{0,t}])$.
Applying Proposition~\ref{int:substitution} on every subinterval gives
\begin{equation}\label{int:associativity}
 \int_\gamma K\,d_{\mathbf X}Q^\gamma
 =\int_\gamma KH\,d_{\mathbf X}J^\gamma
 =\mathcal I_{\mathbf X}(KHY)([\gamma]).
\end{equation}

\begin{theorem}[Tree primitives and finite $\rho$-variation]\label{int:tree}
Suppose $P=T$ is a metric tree and $o\in T$. There is a unique map
$J_Y:T\to E$ with $J_Y(o)=0$ and a finite constant in the estimate
$\|J_Y(y)-J_Y(x)-A^Y_{xy}\|\le C d(x,y)^q$.
It satisfies this estimate with $C=c_\alpha D_Y$.
Let $\eta:[0,1]\to T$ be continuous and
\[
 V_\rho(\eta)^\rho=\sup_D\sum_i
      d(\eta(t_i),\eta(t_{i+1}))^\rho<\infty,
 \qquad 1\le\rho<3\alpha.
\]
With $x_i=\eta(t_i)$ and $\delta_D=\max_i d(x_i,x_{i+1})$,
\begin{equation}\label{int:variation-error}
 \left\|J_Y(\eta(1))-J_Y(\eta(0))-
                         \sum_i A^Y_{x_i x_{i+1}}\right\|
 \le c_\alpha D_Y V_\rho(\eta)^\rho\delta_D^{3\alpha-\rho}
 \longrightarrow0.
\end{equation}
Under the assumptions of Proposition~\ref{int:substitution},
\begin{equation}\label{int:tree-substitution}
 \int_\eta H\,d_{\mathbf X}J_Y
 =\int_\eta HY\,d\mathbf X
 =J_{HY}(\eta(1))-J_{HY}(\eta(0)),
\end{equation}
where the integrals are the limits of their compensated sums.
\end{theorem}
\begin{proof}
For the geodesic $r_{xy}$ write
$I_{xy}=\mathcal I_{\mathbf X}(Y)([r_{xy}])$.
Theorem~\ref{thm:tree}, applied to the translations by $A^Y$, gives
$I_{xy}+I_{yz}=I_{xz}$ and
$\|I_{xy}-A^Y_{xy}\|\le c_\alpha D_Yd(x,y)^q$.
When $D_Y=0$, these follow directly from exact additivity.
Hence $J_Y(x)=I_{ox}$ has all the asserted properties.
For a competing primitive $\widetilde J$ with constant
$\widetilde C$, subdivision of $[x,y]$ into $n$ equal segments gives
\[
 \|(\widetilde J-J_Y)(y)-(\widetilde J-J_Y)(x)\|
 \le(\widetilde C+c_\alpha D_Y)d(x,y)^q n^{1-q}
 \longrightarrow0.
\]
Normalization at $o$ proves uniqueness. Summing the primitive estimate
over $D$ gives \eqref{int:variation-error}, since
$\sum_i d(x_i,x_{i+1})^q
\le\delta_D^{q-\rho}V_\rho(\eta)^\rho$.
Uniform continuity gives $\delta_D\to0$ as $\mesh(D)\to0$.
For substitution, the difference between the sums
\[
 \sum_i\bigl(H_{x_i}(J_Y(x_{i+1})-J_Y(x_i))
                    +B_{x_i}\mathbb X_{x_i x_{i+1}}\bigr),
 \qquad\sum_i A^{HY}_{x_i x_{i+1}}
\]
is bounded by
$\|H\|_\infty c_\alpha D_Y V_\rho(\eta)^\rho
\delta_D^{q-\rho}$.
Applying \eqref{int:variation-error} to $HY$ proves
\eqref{int:tree-substitution}.

For identification with interval integration, put $p=\rho/\alpha\in[1,3)$.
For every $D$ the metric seminorm estimates give
\begin{align*}
 \sum_i\|X_{x_i x_{i+1}}\|^p&\le M_1^pV_\rho(\eta)^\rho,&
 \sum_i\|\mathbb X_{x_i x_{i+1}}\|^{p/2}
 &\le M_2^{p/2}V_\rho(\eta)^\rho,\\
 \sum_i\|R^Y_{x_i x_{i+1}}\|^{p/2}
 &\le[R^Y]_{2\alpha}^{p/2}V_\rho(\eta)^\rho,&
 \sum_i\|Y'_{x_{i+1}}-Y'_{x_i}\|^p
 &\le[Y']_\alpha^pV_\rho(\eta)^\rho.
\end{align*}
By \eqref{int:compact-control}, $Y$ is $\alpha$-H\"older on the compact
set $S=\eta([0,1])$. For $2\le p<3$, these are controlled rough data,
and \eqref{int:variation-error} uses their Gubinelli sums.
For $p<2$, both $Y\circ\eta$ and $x^\eta$ have finite
$p$-variation, and the compensator has norm at most
\[
 \sup_t\|Y'_{\eta(t)}\|M_2V_\rho(\eta)^\rho
                      \delta_D^{2\alpha-\rho}\longrightarrow0.
\]
The limit is therefore the Young integral. For substitution, put
$r=\diam S$ and
$R^{J_Y}_{xy}=J_Y(y)-J_Y(x)-Y_xX_{xy}$. The primitive estimate gives
\begin{align*}
 [R^{J_Y}]_{2\alpha;S}
 &\le\|Y'\|_{\infty;S}M_2+c_\alpha D_Yr^\alpha,\\
 [J_Y]_{\alpha;S}
 &\le\|Y\|_{\infty;S}M_1+\|Y'\|_{\infty;S}M_2r^\alpha
                  +c_\alpha D_Yr^{2\alpha}.
\end{align*}
Thus $(J_Y\circ\eta,Y\circ\eta)$ is a controlled pair for $p\ge2$.
For $p<2$, $H\circ\eta$ and $J_Y\circ\eta$ have finite $p$-variation,
and $\|\sum_i B_{x_i}\mathbb X_{x_ix_{i+1}}\|$ is bounded by
$\|H'\|_{\infty;S}\|Y\|_{\infty;S}M_2
V_\rho(\eta)^\rho\delta_D^{2\alpha-\rho}\to0$.
This identifies \eqref{int:tree-substitution} with common-controller
rough integration or Young integration, respectively.
\end{proof}

This application is only a particular example of the integration theory provided by
metric-groupoid sewing. Further development will be pursued in
subsequent publications.

\end{document}